\documentclass[12pt]{article}
\usepackage{amsfonts,amssymb,amsthm,amsmath,cite,bbm}
\usepackage{enumerate}
\usepackage{enumitem}
\usepackage{hyperref}

\hypersetup{
	colorlinks,%
	citecolor=black,%
	filecolor=black,%
	linkcolor=black,%
	urlcolor=black
}

\newcommand{\N}{\mathbb N} 
\newcommand{\Z}{\mathbb Z} 
\newcommand{\R}{\mathbb R} 
\newcommand{\Rn}{\R^n} 
\newcommand{\Sph}{\mathbb{S}^{n-1}} 

\newcommand{\CV}{\operatorname{Conv}(\Rn)}
\newcommand{\CVc}{\operatorname{Conv_{c}}(\Rn)} 
\newcommand{\CVcn}{\operatorname{Conv}_{\operatorname{c}}^n(\Rn)} 
\newcommand{\CVs}{\operatorname{Conv_{sc}}(\Rn)} 

\newcommand{\K}{{\mathcal K}}
\newcommand{\Kn}{\K^n} 
\newcommand{\Knn}{\K_n^n} 

\newcommand{\Ind}{\mathrm{I}^\infty}

\renewcommand{\d}{\,\mathrm{d}}

\newcommand\SLn{\operatorname{SL}(n)}
\newcommand\GLn{\operatorname{GL}(n)}

\newcommand{\dom}{\operatorname{dom}}
\newcommand{\interior}{\operatorname{int}}
\newcommand{\diam}{\operatorname{diam}}

\newcommand{\bd}{\operatorname{bd}}
\newcommand{\conv}{\operatorname{conv}}
\newcommand{\epi}{\operatorname{epi}}
\newcommand{\infconv}{\mathbin{\Box}}

\newcommand{\eto}{\stackrel{epi}{\longrightarrow}}

\newcommand{\Mom}[2]{M_{#1}^{#2-1}} 
\newcommand{\MomO}{M^{n-1}}
\newcommand{\MomZ}{M^{0}}

\newtheorem{lemma}{Lemma}[section]
\newtheorem{theorem}[lemma]{Theorem}
\newtheorem*{theorem*}{Theorem}
\newtheorem{proposition}[lemma]{Proposition}
\newtheorem{corollary}[lemma]{Corollary}
\newtheorem*{corollary*}{Corollary}

\theoremstyle{definition}

\theoremstyle{remark}
\newtheorem{remark}[lemma]{Remark}
\newtheorem*{remark*}{Remark}
\newtheorem{example}[lemma]{Example}

\let\oldabstract\abstract
\let\oldendabstract\endabstract
\makeatletter
\renewenvironment{abstract}
{%
               {\list{}{\addtolength{\leftmargin}{-0.5em} 
                        \listparindent 1.5em%
                        \itemindent    \listparindent%
                        \rightmargin   \leftmargin%
                        \parsep        \z@ \@plus\p@}%
                \item\relax}%
               {\endlist}%
\oldabstract}
{\oldendabstract}
\makeatother

\title{Metrics and Isometries for Convex Functions}

\author{Ben Li and Fabian Mussnig}
\date{}

\begin{document}

\maketitle

\begin{abstract}
We introduce a class of functional analogs of the symmetric difference metric on the space of coercive convex functions on $\Rn$ with full-dimensional domain. We show that convergence with respect to these metrics is equivalent to epi-convergence. Furthermore, we give a full classification of all isometries with respect to some of the new metrics. Moreover, we introduce two new functional analogs of the Hausdorff metric on the spaces of coercive convex functions and super-coercive convex functions, respectively, and prove equivalence to epi-convergence.
\bigskip

{\noindent
}
\end{abstract}

\section{Introduction}
In the field of convex geometry the focus of interest is on non-empty, compact, convex subsets of $\Rn$, which are also called \emph{convex bodies} and which are commonly denoted by $\Kn$. Usually the space of convex bodies is equipped with the \emph{Hausdorff metric}
$$d_H(K,L)=\inf \{\varepsilon\geq 0\colon K\subseteq L + \varepsilon B^n \text{ and } L \subseteq K + \varepsilon B^n\}$$
for every $K,L\in\Kn$, where $\varepsilon B^n=\{x\in\Rn\colon |x|\leq \varepsilon\}$ denotes the Euclidean ball of radius $\varepsilon\geq 0$ in $\Rn$ and where $C+D=\{x+y\colon x\in C, y\in D\}$ denotes the Minkowski sum of the convex bodies $C,D\in\Kn$. Continuity with respect to the Hausdorff metric is an important property of many operators on convex bodies. For example, the $n$-dimensional volume, $V_n$, is continuous and more generally so are the intrinsic volumes. Also well-known operators that assign a vector to a convex body, e.g., the moment vector, or again a convex body, e.g., the projection body, are continuous with respect to this metric.

For certain applications also other metrics for (subsets of) convex bodies are used. Most importantly, the set of convex bodies with non-empty interiors, $\Knn$, is frequently considered together with the \emph{symmetric difference metric}
$$d_S(K,L)=V_n(K\Delta L)$$
for $K,L\in\Knn$, where $K\Delta L = K\backslash L \cup L\backslash K$ is the symmetric difference of $K$ and $L$. Note that a convex body $K\in\Kn$ has non-empty interior and is therefore an element of $\Knn$ if and only if $V_n(K)>0$. It was shown by Shephard and Webster \cite{shephard_webster_metrics} that the Hausdorff metric and the symmetric difference metric are equivalent on $\Knn$, that is
$$d_H(K_i,K)\to 0 \quad \Longleftrightarrow \quad d_S(K_i,K)\to 0$$
as $i\to\infty$ for every sequence $K_i\in\Knn$ and $K\in\Knn$.

Among the central applications where the symmetric difference metric is used instead of the Hausdorff metric, are best and random approximations of convex bodies, for example by using polytopes with a given number of vertices or facets. We remark that the deep connection of this problem with the affine surface area of a convex body is of particular interest. See, for example, \cite{besau_hoehner_kur_imrn,besau_ludwig_werner_tams_2018,boeroeczky_adv_math_2000,gruber_math_ann_1988,gruber_handbook_approx,gruber_forum_math_1993,ludwig_mathematika_1999,reitzner_adv_math_2005,schneider_studia_sci_math_hungar_1986,schuett_math_nachr_1994}.

\medskip

In recent years, many results from the theory of convex bodies have been extended to convex functions \cite{alesker_adv_geom_2019,alonso_artstein_gonzalez_jimenez_villa_math_ann_2019,artstein_florentin_segal_adv_math_2020,artstein_klartag_schuett_werner_jfa_2012,artstein_milman_annals_2009,artstein_milman_jems_2011,artstein_slomka_pams,artstein_slomka_jga,colesanti_fragala_adv_math_2013,colesanti_lombardi_parapatits_studia_math_2018,colesanti_ludwig_mussnig_mink,colesanti_ludwig_mussnig,colesanti_ludwig_mussnig_jfa_2020,colesanti_ludwig_mussnig_hadwiger,fradelizi_meyer_adv_math_2008,knoerr_j_smooth,knoerr_j_support,li_existence,li_schuett_werner_israel_j_math_2019,li_schuett_werner_j_geom_anal,milman_rotem_jfa_2013,mussnig_adv_math_2019,mussnig_canadian_j_math}.
In particular, the map
\begin{equation}
\label{eq:u_mapsto_int_e_-u}
u\mapsto \int_{\Rn} e^{-u(x)} \d x
\end{equation}
where $u:\Rn\to \R\cup \{+\infty\}$ is a convex function, has been established as a functional analog of the volume of a convex body. For example, by using this analogy the Pr\'ekopa-Leindler inequality \cite{leindler} can be seen as a functional version of the Brunn-Minkowski inequality where the volume is replaced by \eqref{eq:u_mapsto_int_e_-u}. A usual assumption is to consider convex functions $u:\Rn\to\R\cup \{+\infty\}$ that are \emph{proper}, i.e., $u\not\equiv +\infty$, and lower semicontinuous and the space of all such functions will be denoted by $\CV$. Furthermore, for many applications it is also assumed that the functions are \emph{coercive}, i.e., $\lim_{|x|\to\infty} u(x)=+\infty$ which is equivalent to $\int_{\Rn} e^{-u(x)}\d x < +\infty$. We will denote the space of all proper, lower semicontinuous, coercive, convex functions on $\Rn$ by $\CVc$.

The standard topology on the space $\CV$ is induced by epi-convergence. Here, we say that a sequence of convex functions $u_k\in\CV$ is epi-convergent to $u\in\CV$ if for every $x\in\Rn$
\begin{itemize}
    \item $\liminf_{k\to\infty} u_k(x_k)\geq u(x)$ for every sequence $x_k\in\Rn$ such that $x_k\to x$,
    \item $\limsup_{k\to\infty} u_k(x_k)\leq u(x)$ for some sequence $x_k\in\Rn$ such that $x_k\to x$.
\end{itemize}
In this case we will also write $u_k\eto u$. We remark that limits with respect to epi-convergence are always lower semicontinuous. While this notion of convergence, which is sometimes also called $\Gamma$-convergence, has been studied extensively in the fields of convex and variational analysis, it seemingly has only recently found its way into convex geometry and geometric analysis. For recent examples in these areas, where epi-convergence plays a crucial part, we refer to \cite{colesanti_ludwig_mussnig_mink,colesanti_ludwig_mussnig,colesanti_ludwig_mussnig_jfa_2020,colesanti_ludwig_mussnig_hadwiger,knoerr_j_smooth,knoerr_j_support,li_existence}. In particular, on the space $\CVc$ the functional \eqref{eq:u_mapsto_int_e_-u} is continuous with respect to this topology.

The topology described above is in fact metrizable and a precise definition for a corresponding metric has been established in the literature (see Section~\ref{subse:epi_conv} for details). However, this metric seems to be cumbersome for many practical purposes. We will therefore introduce and study new functional analogs of the symmetric difference metric which correspond to $L^p$-metrics of classes of associated quasi-concave functions. For this purpose we restrict to the space
$$\CVcn = \{u:\Rn\to \R\cup \{+\infty\} \colon u \text{ is proper, l.s.c, coercive, convex and}\,\dim \dom u = n\}$$
where $\dom u = \{x\in\Rn\colon u(x)<+\infty\}$ is the \emph{domain} of the function $u$ and $\dim \dom u$ denotes its dimension. Equivalently,
$$u\in\CVcn \quad \Longleftrightarrow \quad u\in\CV \text{ and } 0<\int_{\Rn} e^{-u(x)} \d x < +\infty.$$
Furthermore, for $p\in[1,\infty)$ let
\begin{align*}
\Mom{p}{n}:=\{\zeta:\R\to(0,\infty)\,:\, &\zeta \text{ is continuous, strictly decreasing, }\\
&\zeta^p \text{ has finite moment of order } n-1\}.
\end{align*}
Here, we say that a continuous function $\xi:\R\to [0,\infty)$ has finite moment of order $k\in\N$ if $\int_0^\infty \xi(t) t^{k}\d t < +\infty$. In particular, the function $t\mapsto e^{-t}$ is an element of $\Mom{p}{n}$ for every $p\in[1,\infty)$ and $n\in\N$. For $\zeta\in\Mom{p}{n}$ we now define
$$\delta_{\zeta,p}(u,v) := \|\zeta(u)-\zeta(v)\|_p = \left(\int_{\Rn} |\zeta(u(x))-\zeta(v(x))|^p \d x \right)^{\frac 1p}$$
for every $u,v\in\CVcn$. We will prove the following result.

\begin{theorem}
\label{thm:metric_delta_zeta_p}
For $p\in[1,\infty)$ and $\zeta\in\Mom{p}{n}$ the functional $\delta_{\zeta,p}$ defines a metric on $\CVcn$. Furthermore, convergence with respect to this metric is equivalent to epi-convergence, that is, for every $u_k,u\in\CVcn$ we have 
$$\delta_{\zeta,p}(u_k,u)\to 0 \quad \Longleftrightarrow \quad  u_k \eto u$$
as $k\to\infty$.
\end{theorem}

On the space $\Knn$ equipped with the symmetric difference metric, Gruber \cite{gruber_mathematika_1978} classified all isometries that map again into the same space and thereby characterized measure-preserving affinities. Here, for two metric spaces $(X,d_X)$ and $(Y,d_Y)$ a map $I:(X,d_X)\to(Y,d_Y)$ is said to be an \emph{isometry} if
$$d_Y(I(x_1),I(x_2))=d_X(x_1,x_2)$$
for every $x_1,x_2\in X$. See also \cite{gruber_colloq_1980,gruber_isr_1982,gruber_lettl_1980,schneider_coll_math_1975,weisshaupt_2001} for similar results.

In the special case $p=1$ the metric $\delta_{\zeta,p}$ corresponds to a measure (depending on $\zeta$) of the symmetric difference of the epigraphs of the functions (see Section~\ref{subse:measures} for details). Following ideas of Gruber \cite{gruber_mathematika_1978} and also Cavallina and Colesanti \cite{cavallina_colesanti} we will give a full classification of all isometries $I:(\CVcn,\delta_{\zeta,1})\to(\CVcn,\delta_{\zeta,1})$. For $\zeta\in\Mom{1}{n}$ let
$$\Phi(\zeta):=\{\phi\in\GLn\colon t\mapsto \zeta^{-1}(\zeta(t)/|\det \phi|) \text{ is well-defined and convex on } \R\}.$$
Observe that $\Phi(\zeta)$ is not empty since every $\phi\in\GLn$ with $|\det \phi|=1$ satisfies the necessary conditions.

\begin{theorem}
\label{thm:class_iso}
Let $\zeta\in \Mom{1}{n}$. A map $I:(\CVcn,\delta_{\zeta,1})\to(\CVcn,\delta_{\zeta,1})$ is an isometry if and only if there exist $\phi\in\Phi(\zeta)$ and $x_0\in\Rn$ such that
$$I(u)=f(u\circ \alpha^{-1})$$
for every $u\in\CVcn$, where $\alpha(x)=\phi(x)+x_0$ for $x\in\Rn$ and $f(t)=\zeta^{-1}(\zeta(t)/|\det \phi|)$ for $t\in\R$.
\end{theorem}

In order to prove Theorem~\ref{thm:class_iso} we will use a functional analog of the Blaschke selection theorem which is presented in Section~\ref{subse:blaschke_selection}. Furthermore, we immediately obtain the following result.

\begin{corollary}
A map $I:(\CVcn,\delta_{\zeta,1})\to(\CVcn,\delta_{\zeta,1})$ is an isometry for every $\zeta\in\Mom{1}{n}$ if and only if there exist $\phi\in\GLn$ with $|\det \phi|=1$ and $x_0\in\Rn$ such that
$$I(u)=u\circ \alpha^{-1}$$
for every $u\in\CVcn$, where $\alpha(x)=\phi(x)+x_0$ for $x\in\Rn$.
\end{corollary}

\begin{remark}
Observe that for $\zeta\in\Mom{1}{n}$ every isometry $I:(\CVcn,\delta_{\zeta,1})\to(\CVcn,\delta_{\zeta,1})$ corresponds to an isometry on $\{g=\zeta\circ u\colon u\in\CVcn\}\subset L^1(\Rn)$, equipped with the $L^1$-metric, that maps again into the same space. By Theorem~\ref{thm:class_iso} each such isometry is of the form
$$g \mapsto \frac{g\circ \alpha^{-1}}{|\det \phi|}$$
with $\phi\in\Phi(\zeta)$ and $\alpha(x)=\phi(x)+x_0$ for some $x_0\in\Rn$. We remark that isometries on (subspaces) of $L^p$-spaces have been studied before, see, for example, \cite{hardin_iumj_1981,lamperti_pjm_1958,rudin_iumj_1976,sourour_jfa_1978}. However, due to the special structure of the function space that is considered here, we are not aware of any method that would allow to deduce Theorem~\ref{thm:class_iso} from these results.
\end{remark}

In Section~\ref{se:further_metrics} we introduce two further metrics which can be interpreted as functional analogs of the Hausdorff metric. We will show that convergence with respect to these metrics is equivalent to epi-convergence on the space $\CVc$ and on the subspace of super-coercive functions, $\CVs$, respectively. Here, a function $u:\Rn\to\R\cup \{+\infty\}$ is said to be \emph{super-coercive} if $\lim_{|x|\to\infty} u(x)/|x| = +\infty$ and we write
$$\CVs=\{u\in \CV\colon u \text{ is super-coercive}\}.$$
This function space has been of particular interest recently \cite{colesanti_ludwig_mussnig_jfa_2020,knoerr_j_smooth,knoerr_j_support,mussnig_canadian_j_math}. Among others, new functional analogs of the classical intrinsic volumes were introduced and characterized on this space by using Hessian measures \cite{colesanti_ludwig_mussnig_hadwiger}. We remark that, using the Legendre-Fenchel transform or convex conjugate, the space $\CVs$ corresponds to the space of all convex functions $u:\Rn\to\R$.

\medskip

This article is organized as follows. We collect some basic results about convex bodies, convex functions, epi-convergence and $\mathcal{L}^p$ spaces in Section~\ref{se:preliminaries}. We then discuss the metric $\delta_{\zeta,p}$ and prove Theorem~\ref{thm:metric_delta_zeta_p} in Section~\ref{se:sym_diff_metric_cvx_fcts}. Section~\ref{se:isometries} is devoted to the study of isometries on $(\CVcn,\delta_{\zeta,1})$ and the proof of Theorem~\ref{thm:class_iso}. A discussion of further metrics can be found in Section~\ref{se:further_metrics}.

\section{Preliminaries}
\label{se:preliminaries}
We work in $n$-dimensional Euclidean space $\Rn$. For two elements $x,y\in\Rn$ we denote by $\langle x,y\rangle$ their inner product and by $|x|=\sqrt{\langle x,x \rangle}$ the Euclidean norm of $x$. We write $\Sph=\{x\in\Rn\colon |x|=1\}$ for the unit sphere in $\Rn$ and we denote by
$$B(x_0,r)=\{x\in\Rn\colon |x-x_0|\leq r\}$$
the Euclidean ball of radius $r>0$ with center $x_0\in\Rn$. In particular, $B^n = B(0,1)$. If for some $y\in\Rn$,
$$H(x_0)=\{x\in\Rn \colon \langle x,y \rangle = \langle x_0,y\rangle\}$$
denotes a hyperplane through $x_0\in\Rn$, then we denote by $H^\pm(x_0)$ the corresponding half-spaces. In particular, we write $H^+(x)=\{x\in\Rn\colon \langle x,y\rangle \geq \langle x_0,y\rangle\}$.

For a set $A\subseteq \Rn$, let $\bd A$ and $\interior A$ denote the (topological) boundary and interior of $A$, respectively. Moreover,
\begin{align*}
\diam A &= \sup\{|x-y|\colon x,y\in A\} \in [0,+\infty]\\
\conv A &= \left\{\sum_{i=1}^m \lambda_i x_i \colon m\in\N, x_i\in A, \lambda_i \geq 0, \sum_{i=1}^m \lambda_i=1 \right\}
\end{align*}
will denote the \emph{diameter} and \emph{convex hull} of $A\subseteq \Rn$.

\subsection{Convex Bodies and Convex Functions}
We collect some elementary facts on convex bodies and convex functions. We use \cite{rockafellar_wets} and \cite{schneider} as standard references.

\medskip

For every $u\in\CVc$ and $t\in\R$ the \emph{sublevel sets}
$$\{u\leq t\} = \{x\in\Rn\colon u(x)\leq t\}$$
are either convex bodies or empty sets. In particular, for every $u\in\CVcn$ and $t>\min_{x\in\Rn} u(x)$ we have
$$\{u\leq t\} \in\Knn.$$
Similarly, if $\zeta:\R\to[0,\infty)$ is decreasing we denote by
$$\{\zeta\circ u\geq s\} = \{x\in\Rn\colon \zeta(u(x))\geq s\}$$
the \emph{superlevel set} of the quasi-concave function $\zeta\circ u$ at $s\geq 0$.
Furthermore, every convex function $u\in\CV$ is uniquely determined by its \emph{epigraph}
$$\epi u =\{(x,t)\in\Rn\times \R \colon u(x)\leq t\},$$
which is a convex subset of $\Rn\times \R = \R^{n+1}$.

For $u\in\CV$ and $x\in\Rn$, we say that $y\in\Rn$ is a \emph{subgradient} of $u$ at $x$ if
$$u(z)\geq u(x)+ \langle z-x, y \rangle$$
for every $z\in\Rn$ and we denote by $\partial u(x)$ the \emph{subdifferential} of $u$ at $x$, which is the set of all subgradients of $u$ at $x$. The subdifferential $\partial u(x)$ is always a closed, convex subset of $\Rn$ and it is non-empty for every $x\in\dom u$. Furthermore, we write
$$u^*(y)=\sup\nolimits_{x\in\Rn} \big(\langle x,y \rangle -u(x) \big)\in \R \cup \{+\infty\}$$
$y\in\Rn$, for the \emph{convex conjugate} or \emph{Legendre-Fenchel transform} of $u\in\CV$. Since $u$ is proper and lower semicontinuous, it follows that also $u^*\in\CV$ and furthermore $u^{**}=u$. 

For two convex functions $u,v\in\CV$ their \emph{infimal convolution} $u\infconv v$ at $x\in\Rn$ is defined as
$$(u\infconv v)(x)=\inf\nolimits_{x=y+z} \big( u(y)+v(z)\big)$$
or equivalently, 
$$\epi (u\infconv v) = \epi u + \epi v,$$
where the sum on the right-hand side denotes the Minkowski sum of convex sets in $\R^{n+1}$. If furthermore $u\infconv v > -\infty$, then
$$(u\infconv v)^* = u^* + v^*.$$

\medskip

We associate two convex functions with every convex body. By
$$\Ind_K(x)=\begin{cases} 0\quad &\text{if } x\in K\\
+\infty \quad &\text{if } x\not\in K\end{cases}
$$
we denote the \emph{(convex) indicator function} of $K\in\Kn$ and remark that $\Ind_K\in\CVc$ for every $K\in\Kn$ and furthermore $\Ind_K\in\CVcn$ if and only if $K\in\Knn$. Moreover, we write
$$h(K,x)=\max\nolimits_{y\in K} \langle x,y \rangle$$
for the \emph{support function} of $K\in\Kn$ at $x\in\Rn$. It is easy to see that support functions of convex bodies are uniquely determined by their values on $\Sph$. Furthermore, $h(K,\cdot)=(\Ind_K)^*$ and $h(K,\cdot)^*=\Ind_K$ for every $K\in\Kn$.

\medskip

We will need the following characterizations of coercive and super-coercive convex functions.

\begin{lemma}[\!\!\cite{rockafellar_wets}, Theorem 11.8]
\label{le:coercive_super_coercive_duals}
A function $u\in\CV$ is coercive if and only if $0\in\interior \dom u^*$. Furthermore, $u$ is super-coercive if and only if $\dom u^*=\Rn$.
\end{lemma}

\begin{lemma}[\!\!\cite{rockafellar_wets}, Theorem 3.26]
\label{le:coercive_super_coercive}
A function $u\in\CV$ is coercive if and only if there exist $a>0$ and $b\in\R$ such that
\begin{equation}
\label{eq:cone_coercive_super_coercive}
u(x)\geq a|x|+b
\end{equation}
for every $x\in\Rn$. Furthermore, $u$ is super-coercive if and only if for every $a>0$ there exists $b\in\R$ such that \eqref{eq:cone_coercive_super_coercive} holds.
\end{lemma}

\subsection{Epi-Convergence}
\label{subse:epi_conv}

As described in the monograph of Rockafellar and Wets \cite{rockafellar_wets}, there are natural ways to define metrics that induce epi-convergence. This comes from the fact that epi-convergence is equivalent to a set-type convergence of the corresponding epigraphs.

For two closed sets $C,D\subseteq \R^{n+1}$ and $\rho \geq 0$, let
$$d_{\rho}(C,D)=\max\nolimits_{|x|\leq \rho} \left|\min\nolimits_{y\in C} |x-y| - \min\nolimits_{y\in D}|x-y|\right|.$$
and let
$$d(C,D)=\int_0^{+\infty} d_{\rho}(C,D)e^{-\rho} \d \rho.$$

\begin{theorem}[\!\!\cite{rockafellar_wets}, Theorem 4.42, Theorem 7.58]
For $\rho \geq 0$, $d_\rho$ defines a pseudo-metric, while $d$ defines a metric on the set $\{C\subset \R^{n+1}\,:\, C \text{ is closed, } C \neq \emptyset\}$. This metric space is complete. Furthermore, for every $u_k,u\in\CV$ we have
$$u_k\eto u \quad \Longleftrightarrow \quad d(\epi u_k,\epi u)\to 0$$
as $k\to\infty$.
\end{theorem}

On $\CVc$, epi-convergence is also closely related to Hausdorff-convergence of the level sets. In the following we say that $\{u_k\leq t \} \to \emptyset$ as $k\to\infty$ if there exists $k_0\in\N$ such that $\{u_k\leq t \}=\emptyset$ for all $k\geq k_0$.

\begin{lemma}[\!\!\cite{colesanti_ludwig_mussnig}, Lemma 5 and \cite{beer_rockafellar_wets}, Theorem 3.1]
\label{le:hd_conv_lvl_sets}
Let $u_k,u\in\CVc$. If $u_k\eto u$, then $\{u_k\leq t \}\to \{u\leq t\}$ as $k\to\infty$ for every $t\neq \min_{x\in\Rn} u(x)$. Conversely, if for every $t\in\R$ there exists a sequence $t_k\to t$ such that $\{u_k\leq t_k\}\to \{u\leq t\}$, then $u_k\eto u$ as $k\to\infty$.
\end{lemma}

We will also need the following basic facts about epi-convergence.

\begin{lemma}[\!\!\cite{rockafellar_wets}, Proposition 7.4 and Theorem 7.17]
\label{le:epi_lims_lsc_cvx}
If $u_k\in\CV$ is epi-convergent to some function $u:\Rn\to\R\cup\{+\infty,-\infty\}$, then also $u$ is lower semicontinuous and convex.
\end{lemma}

\begin{lemma}[\!\!\cite{rockafellar_wets}, Theorem 7.17]
\label{le:epi_conv_equiv}
Let $u_k,u\in\CVc$. If $u\in\CVcn$, then the following are equivalent:
\begin{itemize}
	\item $u_k\eto u$,
	\item there exists a dense subset $D\subseteq \Rn$ such that $u_k(x)\to u(x)$  for every $x\in D$,
	\item $u_k$ converges to $u$ uniformly on every compact set $C\subset \Rn$ such that $C\cap \bd \dom u=\emptyset$.
\end{itemize}
\end{lemma}

\begin{remark}
\label{re:ptw_unif_conv_cvx}
We remark, that if a sequence $u_k\in\CV$ converges pointwise to a function $u\in\CVcn$, then the sequence already converges uniformly on every compact subset of $\interior\dom u$.
\end{remark}

\begin{theorem}[\!\!\cite{rockafellar_wets}, Theorem 11.34]
\label{thm:epi_conv_conjugate}
If $u_k,u\in\CV$, then $u_k\eto u$ if and only if $u_k^*\eto u^*$ as $k\to\infty$.
\end{theorem}

The following lemma tells us that the epigraphs of a epi-convergent sequence in $\CVc$ are contained in a pointed cone. Hence, one may also call this a \emph{uniform cone property}.
\begin{lemma}[\!\!\cite{colesanti_ludwig_mussnig}, Lemma 8]
\label{le:uniform_cone}
If $u_k,u\in\CVc$ are such that $u_k\eto u$ as $k\to\infty$, then there exist $a>0$ and $b\in\R$ such that
$$u_k(x)\geq a|x|+b\quad \text{and} \quad u(x)\geq a|x|+b$$
for every $x\in\Rn$ and $k\in\N$.
\end{lemma}

We conclude this section with some functionals on $\CVc$ that are continuous with respect to epi-convergence.

\begin{lemma}[\!\!\cite{colesanti_ludwig_mussnig}, Lemma 12]
\label{le:min_cont}
If $u_k,u\in\CVc$ are such that $u_k\eto u$, then\linebreak$\min_{x\in\Rn} u_k(x) \to \min_{x\in\Rn}u(x)$ as $k\to\infty$.
\end{lemma}

\begin{lemma}[\!\!\cite{colesanti_ludwig_mussnig}, Lemma 14]
\label{le:int_zeta_finite}
Let $\zeta:\R\to[0,\infty)$ be continuous. For every $u\in\CVc$
$$\int_{\dom u} \zeta(u(x)) \d x < +\infty$$
if and only if $\zeta$ has finite moment of order $n-1$.
\end{lemma}

\begin{lemma}[\!\!\cite{colesanti_ludwig_mussnig}, Lemma 15]
\label{le:int_continuous}
If $\zeta:\R\to[0,\infty)$ is continuous with finite moment of order $n-1$, then
$$u\mapsto \int_{\dom u} \zeta(u(x)) \d x$$
is continuous on $\CVc$.
\end{lemma}

\subsection{A Selection Theorem}
\label{subse:blaschke_selection}
The aim of this section is to formulate functional analogs of the following well-known result (see, for example, \cite[Theorem 1.8.7]{schneider}).
\begin{theorem}[Blaschke selection theorem]
\label{thm:blaschke_selection_cvx_bds}
Every bounded sequence of convex bodies has a subsequence that converges to a convex body.
\end{theorem}

Our result will essentially follow from the next proposition.

\begin{proposition}[\!\!\cite{rockafellar_wets}, Theorem 7.6]
\label{prop:conv_subseq}
Let $u_k:\Rn\to\R\cup \{+\infty\}$ be a sequence and let $r>0$. If for every $k\in\N$ there exists $x_k\in r B^n$ such that $u_k(x_k)< r$, then there exists a subsequence $u_{k_j}$ that is epi-convergent to a function $u:\Rn\to \R\cup\{-\infty,+\infty\}$ with $u\not\equiv +\infty$.
\end{proposition}

\begin{theorem}[Selection theorem]
\label{thm:blaschke_selection_cvx_fcts}
Let $u_k\in\CVc$ be sequence. If there there exist $a>0$ and $b,m\in\R$ such that
$$\min\nolimits_{x\in\Rn} u_k(x)\leq m$$
and
\begin{equation}
\label{eq:sel_them_unif_cone}
u_k(x)\geq a|x|+b
\end{equation}
for every $x\in\Rn$ and every $k\in\N$, then there exists a subsequence $u_{k_j}$ that is epi-convergent to a function $u\in\CVc$. If furthermore for every $a>0$ there exists $b\in\R$, possibly depending on $a$, such that \eqref{eq:sel_them_unif_cone} holds, then $u\in\CVs$.
\end{theorem}
\begin{proof}
For $k\in\N$ let $x_k\in\Rn$ be such that $u_k(x_k)=\min_{x\in\Rn}u_k(x)$. By our assumptions we now have
$$a|x_k|+b \leq u_k(x_k) \leq m$$
for every $k\in\N$ and therefore
$$|x_k|\leq \frac{m-b}{a}.$$
Thus, the sequence $u_k$ satisfies the conditions of Proposition~\ref{prop:conv_subseq} with $r=\max\{m,(m-b)/a\}$. Hence, there exists a subsequence $u_{k_j}$ that is epi-convergent to some function $u:\Rn\to \R\cup\{-\infty,+\infty\}$ with $u\not\equiv +\infty$. By Lemma~\ref{le:epi_lims_lsc_cvx} the function $u$ is convex and lower semicontinuous. Furthermore, it easily follows from the definition of epi-convergence and \eqref{eq:sel_them_unif_cone} that also the function $u$ satisfies
$$u(x)\geq a|x|+b$$
for every $x\in\Rn$, which together with Lemma~\ref{le:coercive_super_coercive} implies that $u$ is coercive. Hence, we conclude that $u\in\CVc$. Similarly, if for every $a>0$ there exists $b\in\R$ such that \eqref{eq:sel_them_unif_cone} holds, then also $u$ has this property and therefore, by Lemma~\ref{le:coercive_super_coercive}, $u\in\CVs$.
\end{proof}

\subsection{\texorpdfstring{$\mathcal{L}^p$}{Lp}-Spaces}
For $p\in[1,\infty)$ we denote by
$$\mathcal{L}^p(\Rn)=\left\{f:\Rn\to\R\colon f \text{ is measurable and } \int_{\Rn} |f(x)|^p\d x < +\infty \right\}$$
and furthermore
$$\|f\|_p = \left(\int_{\Rn} |f(x)|^p\d x \right)^{\frac 1p}.$$
Note, that the functional $\|\cdot\|_p$ only defines a seminorm on $\mathcal{L}^p$, i.e., it need not be positive definite.

\medskip

We will need the following convergence result due to Riesz \cite{riesz}, which for the case $p=1$ is also known as Scheff\'e's lemma \cite{scheffe}. See also \cite{kusolitsch}. We remark that those results are usually formulated for the normed space $(L^p(\Rn),\|.\|_p)$, which is obtained in the usual way as a quotient space of $\mathcal{L}^p$. However, it is easy to see that they are also valid for the seminormed space $(\mathcal{L}^p(\Rn),\|.\|_p)$.

\begin{lemma} \label{le:pth_Scheffe}
Let $p\in[1,\infty)$. If $f_k,f\in \mathcal{L}^p(\Rn)$ are such that $f_k\to f$ pointwise a.e., then $\|f_k-f\|_p\to 0$ if and only if $\|f_k\|_p\to \|f\|_p$.
\end{lemma}

\section{Functional Analogs of the Symmetric Difference Metric}
\label{se:sym_diff_metric_cvx_fcts}
In this section we discuss the properties of the metric $\delta_{\zeta,p}$ and present the proof of Theorem~\ref{thm:metric_delta_zeta_p}, which split into three separate statements. In Lemma~\ref{le:delta_zeta_p_is_a_metric} we show that $\delta_{\zeta,p}$ defines a metric if $\zeta\in\Mom{p}{n}$. We prove that epi-convergence implies convergence with respect to $\delta_{\zeta,p}$ in Lemma~\ref{le:epi_conv_implies_int_metric_conv}. Lastly, we show that convergence with respect to $\delta_{\zeta,p}$ implies epi-convergence in Proposition~\ref{prop:int_metric_conv_implies_epi_conv}.

\bigskip

A naive way to measure the distance between two functions $u,v\in\CVc$ is to integrate over distances of their level sets, which are either empty or convex bodies. In particular, one might consider the expression
$$\int_{\R} V_n(\{u\leq t\} \Delta \{v\leq t\}) \xi(t) \d t$$
where $\xi:\R\to[0,\infty)$ is chosen in such a way that this integral is finite. If additionally $\xi = -\zeta'$ for some decreasing function $\zeta:\R\to[0,\infty)$, then this can be rewritten as
\begin{align*}
\int_{\R} V_n(\{u\leq t\} \Delta \{v\leq t\}) (-\zeta'(t)) \d t &= \int_{\R} \big(V_n(\{(u\wedge v)\leq t\})-V_n((u\vee v)\leq t\})\big) (-\zeta'(t))\d t\\
&= \int_{0}^{\infty} \big( V_n(\{\zeta(u\wedge v) \geq s\})-V_n(\{\zeta(u\vee v) \geq s\})\big) \d s\\
&= \int_{\Rn} \big(\zeta((u\wedge v)(x))-\zeta((u\vee v)(x))\big) \d x\\
&= \int_{\Rn} |\zeta(u(x))-\zeta(v(x))|\d x.
\end{align*}
This motivates to define $\delta_{\zeta,p}:\CVc\times \CVc \to [0,\infty)$ as
\begin{equation}
\label{eq:def_delta_zeta_p}
\delta_{\zeta,p}(u,v)=\left(\int_{\Rn} |\zeta(u(x))-\zeta(v(x))|^p\d x \right)^{\frac{1}{p}}
\end{equation}
for every $u,v\in\CVc$, where $p\in[1,\infty)$ and where $\zeta:\R\to[0,\infty)$ is a continuous and decreasing function. Observe, that
$$\delta_{\zeta,p}(u,v)\leq \int_{\Rn} \zeta^p(u(x)) \d x + \int_{\Rn} \zeta^p(v(x)) \d x$$
for every $u,v\in\CVc$ and furthermore
$\delta_{\zeta,p}(u,v)=\int_{\Rn} \zeta(u(x))^p \d x$, if $\dim \dom v <n$. Thus, by Lemma~\ref{le:int_zeta_finite}, it is sufficient and necessary to assume that $\zeta^p$ has finite moment of order $n-1$ if we want to ensure that $\delta_{\zeta,p}(u,v)$ is finite for every $u,v\in\CVc$.

\begin{remark}
Observe that in general \eqref{eq:def_delta_zeta_p} is not well-defined, since $\zeta(u(x))$ is not defined if $x\not\in \dom u$. However, we will always assume that $\zeta$ is decreasing and that $\zeta^p$ has finite moment of order $n-1$, which implies that $\lim_{t\to\infty} \zeta(t)=0$. Hence, for $x\not\in \dom u$ or equivalently $u(x)=+\infty$ we will therefore write $\zeta(u(x)):=\lim_{t\to\infty} \zeta(t)=0$. In particular, by using this notation we do not need to restrict the domain of integration in the definition of $\delta_{\zeta,p}$.
\end{remark}

\begin{example}
For $u=\Ind_K$ and $v=\Ind_L$ with $K,L\in \Kn$ we have
$$ |\zeta(u(x))-\zeta(v(x))| = \begin{cases}
\zeta(0)\quad &\text{if } x\in K\Delta L\\
0\quad &\text{else.} 
\end{cases}
$$
Thus, $\delta_{\zeta,p}(\Ind_K,\Ind_L)=\zeta(0) (V_n(K\Delta L))^{1/p}$.
\end{example}

\begin{lemma}
\label{le:epi_conv_implies_int_metric_conv}
Let \(p\in [1,\infty)\) and let $\zeta:\R\to[0,\infty)$ be continuous and decreasing such that $\zeta^p$ has finite moment of order $n-1$. If $u_k,u\in\CVc$ are such that $u_k\eto u$, then $\delta_{\zeta,p}(u_k,u)\to 0$ as $k\to\infty$.
\end{lemma}
\begin{proof}
Since $\zeta^p$ is continuous and has finite moment of order $n-1$, it follows from Lemma~\ref{le:int_zeta_finite} that $\zeta\circ u_k, \zeta\circ u\in \mathcal{L}^p(\Rn)$. Furthermore, by Lemma~\ref{le:int_continuous} 
$$\lim\nolimits_{k\to\infty} \int_{\Rn} \zeta^p(u_k(x)) \d x = \int_{\Rn} \zeta^p(u(x)) \d x.$$
If $\dim \dom u<n$, then $\int_{\Rn}\zeta(u(x))\d x = 0$ and thus,
$$\lim\nolimits_{k\to\infty} \delta_{\zeta,p}(u_k,u)=\lim\nolimits_{k\to\infty} \int_{\Rn} \zeta^p(u_k(x)) \d x=0.$$
In case $\dom \dom u=n$, it follows from the continuity of $\zeta$ and Lemma~\ref{le:epi_conv_equiv} that $\zeta\circ u_k$ converges to $\zeta\circ u$ pointwise almost everywhere. Hence, by Lemma~\ref{le:pth_Scheffe},
$$\lim\nolimits_{k\to\infty} \int_{\Rn} |\zeta(u_k(x))-\zeta(u(x))|^p \d x=\lim\nolimits_{k\to\infty} \|\zeta\circ u_k - \zeta\circ u\|_p^p = 0$$
which is equivalent to $\delta_{\zeta,p}(u_k,u)\to 0$.
\end{proof}

\begin{remark}
The last statement cannot be extended to the case $p=\infty$, in general. For example, if $u=\Ind_{B^n}+t$ and $u_k=\Ind_{(1-1/k)B^n}+t$, where $t\in\R$ is such that $\zeta(t)>0$, then $u_k\eto u$ by Lemma~\ref{le:hd_conv_lvl_sets} but
$$|\zeta^p(u(x))-\zeta^p(u_k(x))|=\zeta^p(t)>0$$
for every $k\in\N$ and every $x\in\{x\in\Rn\colon (1-1/k)<|x|\leq 1 \}$, which is a set of positive measure.
\end{remark}

It is easy to see that in order for $\delta_{\zeta,p}$ to become a metric and furthermore induce epi-convergence, two things are necessary: $\zeta$ needs to be strictly decreasing and positive and one needs to restrict to functions with full-dimensional domain. Recall that for $p\in [1,\infty)$ and $n\in\N$ we set
\begin{align*}
\Mom{p}{n}=\{\zeta:\R\to(0,\infty)\,:\, &\zeta \text{ is continuous, strictly decreasing,}\\
&\zeta^p \text{ has finite moment of order } n-1\}.
\end{align*}

\begin{remark}
Let $p\in[1,\infty)$, $n\in\N$ and $\zeta\in\Mom{p}{n}$.
Observe that for $q\geq p$ we have
\begin{align*}
\int_0^\infty \zeta^q(t) t^{n-1} \d t &=\int_{\{t\geq 0 \colon \zeta(t)\geq 1\}} \zeta^q(t) t^{n-1}\d t + \int_{\{t\geq 0 \colon \zeta(t)< 1\}} \zeta^q(t) t^{n-1} \d t\\
&\leq \int_{\{t\geq 0 \colon \zeta(t)\geq 1\}} \zeta^q(t) t^{n-1}\d t + \int_{\{t\geq 0 \colon \zeta(t)< 1\}} \zeta^p(t) t^{n-1} \d t
\end{align*}
Since $\zeta$ is continuous and strictly decreasing with $\lim_{t\to\infty}\zeta(t)=0$ the set $\{t\geq 0\colon \zeta(t)\geq 1\}$ is either a compact interval or empty. Thus, we conclude that
$$\int_0^{\infty} \zeta^q(t)t^{n-1} < +\infty$$
and therefore $\zeta\in\Mom{q}{n}$. In particular, $\Mom{p}{n}\subseteq \Mom{q}{n}$ whenever $1\leq p \leq q < \infty$.
\end{remark}

\begin{lemma}
\label{le:delta_zeta_p_is_a_metric}
Let $p\in [1,\infty)$. If $\zeta\in \Mom{p}{n}$, then the functional $\delta_{\zeta,p}$ defines a metric on $\CVcn$.
\end{lemma}
\begin{proof}
Since $\zeta^p$ has finite moment of order $n-1$ it follows from Lemma~\ref{le:int_zeta_finite} that $\zeta\circ u\in \mathcal{L}^p(\Rn)$ for every $u\in\CVcn$. Furthermore, $\delta_{\zeta,p}(u,v)=\|\zeta\circ u - \zeta\circ v\|_p$
for every $u,v\in\CVcn$. Hence, it follows from the properties of the seminormed space $(\mathcal{L}^p(\Rn),\|\cdot\|_p)$ that $\delta_{\zeta,p}(u,v)<+\infty$, $\delta_{\zeta,p}(u,v)=\delta_{\zeta,p}(v,u)$, $\delta_{\zeta,p}(u,u)=0$ and $\delta_{\zeta,p}(u,v)\leq \delta_{\zeta,p}(u,w)+\delta_{\zeta,p}(w,v)$ for every $u,v,w\in\CVcn$. It remains to show that if $u,v\in\CVcn$ are such that $\delta_{\zeta,p}(u,v)=0$, then $u\equiv v$. Therefore let such $u,v,\in\CVcn$ be given. Note, that by the properties of $\Mom{p}{n}$ we have $\zeta(t)>0$ for every $t\in\R$ and $\lim_{t\to\infty}\zeta(t)=0$. Thus,
\begin{align*}
0&=\int_{\Rn} |\zeta(u(x))-\zeta(v(x))|^p\d x\\
&= \int_{\dom u \cap \dom v} |\zeta(u(x))-\zeta(v(x))|^p\d x + \int_{\dom u \backslash \dom v} \zeta^p(u(x))\d x\\
&\quad + \int_{\dom v \backslash \dom u} \zeta^p(v(x))\d x.
\end{align*}
Since each of the terms in this equation is non-negative and since $\zeta^p\circ u$ is strictly positive on $\dom u \backslash \dom v$, this implies that $\dom u \backslash \dom v$ is at most $(n-1)$-dimensional. Similarly, also $\dom v \backslash \dom u$ is at most $(n-1)$-dimensional. Hence, since both $\dom u$ and $\dom v$ are $n$-dimensional convex sets also the set $\dom u \cap \dom v$ must be $n$-dimensional. Since both $\zeta\circ u$ and $\zeta\circ v$ are continuous on the $n$-dimensional set $\dom u \cap \dom v$, we must have $\zeta\circ u \equiv \zeta\circ v$ on $\dom u = \dom v$. As $\zeta$ is strictly decreasing it is invertible and together with the lower semicontinuity of both $u$ and $v$ we conclude that $u\equiv v$.
\end{proof}

We will use the following trivial result.

\begin{lemma}
\label{le:halfspace}
Let $u\in\CV$. For every $x_0\in \Rn$ there exists a hyperplane $H(x_0)$ through $x_0$ such that
$$u(x)\geq u(x_0)$$
for all $x\in H^+(x_0)$.
\end{lemma}
\begin{proof}
The statement is trivial if $x_0 \notin \interior \dom u$. If $x_0\in\interior \dom u$, then there exists $y\in\partial u(x_0)$ and therefore
$$u(x)\geq u(x_0)+\langle x-x_0, y \rangle \geq u(x_0)$$
for every $x\in H^+(x_0):=\{x\in\Rn\colon \langle x,y \rangle \geq \langle x_0,y \rangle \}$.
\end{proof}

\begin{lemma}
\label{le:pointw_conv_in_int_dom}
Let $p\in [1,\infty)$ and $\zeta\in\Mom{p}{n}$. If $u_k,u\in\CVcn$ are such that $\delta_{\zeta,p}(u_k,u)\to 0$, then $u_k(x_0)\to u(x_0)$ as $k\to\infty$ for every $x_0\in \interior \dom u$.
\end{lemma}
\begin{proof}
Assume on the contrary, that there exists $x_0\in\interior \dom u$ such that $u_k(x_0) \not\to u(x_0)$. Hence, there exist $\varepsilon>0$ and a subsequence $u_{k_j}$ of $u_k$ such that $|u_{k_j}(x_0)-u(x_0)|>\varepsilon$ for every $j\in\N$. By possibly restricting to another subsequence we can either assume that $u_{k_j}(x_0)>u(x_0)+\varepsilon$ or $u(x_0)-\varepsilon>u_{k_j}(x_0)$ for every $j\in\N$.

Let us first consider the case $u_{k_j}(x_0)>u(x_0)+\varepsilon$. Since $x_0\in\interior \dom u$ there exists $R>0$ such that $B(x_0,R)\subset \interior \dom u$ and
$$u(x)<u(x_0)+\frac{\varepsilon}{2}$$
for every $x\in B(x_0,R)$. By Lemma~\ref{le:halfspace} there exist hyperplanes $H_j(x_0)$ through $x_0$ such that
$$u_{k_j}(x)\geq u_{k_j}(x_0) > u(x_0)+\varepsilon$$
for every $x\in H_j^+(x_0)$ and $j\in\N$.  Note, that $C:= V_n(H_j^+(x_0)\cap B(x_0,R)) >0$ does not depend on $j$. Hence,
\begin{align*}
\int_{\Rn}|\zeta(u_{k_j}(x))-\zeta(u(x))|^p \d x &\geq \int_{H_j^+(x_0) \cap B(x_0,R)} \left(\zeta(u(x))-\zeta(u_{k_j}(x))\right)^p \d x\\
&\geq C \left(\zeta\big(u(x_0)+\tfrac{\varepsilon}{2}\big)-\zeta\big(u(x_0)+\varepsilon\big)\right)^p\\
&>0
\end{align*}
which contradicts $\delta_{\zeta,p}(u_k,u)\to 0$.

Next, consider the case $u_{k_j}(x_0)<u(x_0)-\varepsilon$. Similar to the first case, let $R>0$ be such that $B(x_0,R)\subset \interior \dom u$ and
\begin{equation}
\label{eq:dist_u_x_u_x_0}
|u(x)-u(x_0)|<\frac{\varepsilon}{4}
\end{equation}
for every $x\in B(x_0,R)$. We claim that there exist $j_0\in\N$ and $x_j\in B(x_0,\tfrac{R}{8})$ such that
\begin{equation}
\label{eq:u_kj_x_j}
u_{k_j}(x_j)> u(x_0)-\frac{\varepsilon}{2}
\end{equation}
for every $j\geq j_0$. If not, we could find another sub-sequence $u_{k_{j_l}}$ such that
$$
u_{k_{j_l}}(x)\leq u(x_0)-\frac{\varepsilon}{2}
$$
for every $x\in B(x_0,\tfrac{R}{8})$, which together with \eqref{eq:dist_u_x_u_x_0} would obviously contradict our initial assumption $\delta_{\zeta,p}(u_{k},u)\to 0$. Hence, let $x_j\in B(x_0,\tfrac{R}{8})$ be as in \eqref{eq:u_kj_x_j} and set
$$\bar{x}_j:=x_0+\frac{R}{2} \frac{x_j-x_0}{|x_j-x_0|}$$
for $j\geq j_0$. By convexity we have
$$\frac{u_{k_j}(\bar{x}_j)-u_{k_j}(x_j)}{|\bar{x}_j-x_j|}\geq \frac{u_{k_j}(x_j)-u_{k_j}(x_0)}{|x_j-x_0|}$$
and therefore
\begin{align*}
u_{k_j}(\bar{x}_j) &\geq \frac{u_{k_j}(x_j)-u_{k_j}(x_0)}{|x_j-x_0|} |\bar{x}_j-x_j| + u_{k_j}(x_j)\\
&> \frac{(u(x_0)-\varepsilon/2)-(u(x_0)-\varepsilon)}{R/8}\left(\frac{R}{2}-\frac{R}{8}\right)+u(x_0)-\frac{\varepsilon}{2}\\
&=u(x_0)+\varepsilon
\end{align*}
for every $j\geq j_0$. By Lemma~\ref{le:halfspace} there exist hyperplanes $H_j(\bar{x}_j)$ through $\bar{x}_j$ such that
$$u_{k_j}(x) \geq u_{k_j}(\bar{x}_j) > u(x_0)+\varepsilon$$
for every $x\in H_j^+(\bar{x}_j)$ and every $j\geq j_0$. On the other hand, since $|\bar{x}_j-x_0|=\frac{R}{2}$ we have $B(\bar{x}_j,\frac{R}{2})\subset B(x_0,R)$ and therefore it follows from \eqref{eq:dist_u_x_u_x_0} that
$$u(x)<u(x_0)+\frac{\varepsilon}{4}$$
for every $x\in B(\bar{x}_j,\frac{R}{2})$. Thus, proceeding similarly as in the first case with $C:=V_n( H_j^+(\bar{x}_j) \cap B(\bar{x}_j,\frac{R}{2}))>0$, this leads to a contradiction of $\delta_{\zeta,p}(u_k,u)\to 0$.
\end{proof}

\begin{lemma}
\label{le:x_0_outside_int_dom}
Let $p\in [1,\infty)$ and $\zeta\in \Mom{p}{n}$. If $u_k,u\in\CVcn$ are such that $\delta_{\zeta,p}(u_k,u)\to 0$, then $u_k(x_0)\to +\infty$ as $k\to\infty$ for every $x_0 \in \interior (\Rn \backslash \dom u)$.
\end{lemma}
\begin{proof}
Assume on the contrary that there exists $x_0 \in \interior (\Rn \backslash \dom u)$ such that $u_k(x_0) \not\to +\infty$. Hence, there exist $m_0\in\R$ and a sub-sequence $u_{k_j}$ of $u_k$ such that $u_{k_j}(x_0)<m_0$. Since $\dim \dom u = n$, there exists a compact set $K\subset \interior \dom u$ with $\dim K=n$. Since $u$ is continuous on $K$ there exists $m_1:=\max_{x\in K} u(x)$. Furthermore, by Lemma~\ref{le:pointw_conv_in_int_dom}, $u_{k_j}(x) \to u(x)$ for every $x\in K$ which by convexity already implies uniform convergence of $u_{k_j}$ to $u$ on $K$ (see also Remark~\ref{re:ptw_unif_conv_cvx}). Hence, there exist $\varepsilon>0$ and $j_0\in\N$ such that
$$u_{k_j}(x)\leq m_1+\varepsilon$$
for every $x\in K$ and $j>j_0$. Let
$$A:=\conv\{x_0,K\}\backslash \dom u.$$
Since $x_0 \in \interior (\Rn \backslash \dom u)$ we have $V_n(A)>0$ and furthermore, by convexity,
$$u_{k_j}(x)\leq \max\{m_0,m_1+\varepsilon\}$$
for every $x\in A$ and $j\geq j_0$. We now have
\begin{align*}
\int_{\Rn}|\zeta(u_{k_j}(x))-\zeta(u(x))|^p \d x &\geq \int_{A}|\zeta(u_{k_j}(x))-\zeta(u(x))|^p\d x\\
&=\int_{A} \zeta^p(u_{k_j}(x)) \d x\\
&\geq \min\{\zeta^p(m_0),\zeta^p(m_1+\varepsilon)\} V_n(A)\\
&>0
\end{align*}
which contradicts $\delta_{\zeta,p}(u_k,u)\to 0$.
\end{proof}

By combining Lemma~\ref{le:epi_conv_equiv}, Lemma~\ref{le:pointw_conv_in_int_dom} and Lemma~\ref{le:x_0_outside_int_dom} we obtain the following result.
\begin{proposition}
\label{prop:int_metric_conv_implies_epi_conv}
Let $p\in [1,\infty)$ and $\zeta\in \Mom{p}{n}$. If $u_k,u\in\CVcn$ are such that $\delta_{\zeta,p}(u_k,u)\to 0$, then $u_k\eto u$ as $k\to \infty$.
\end{proposition}

\section{Isometries}
\label{se:isometries}
In this section we study isometries $I:(\CVcn,\delta_{\zeta,1})\to(\CVcn,\delta_{\zeta,1})$ with $\zeta\in\Mom{1}{n}$, that is,
$$\delta_{\zeta,1}(I(u),I(v)) = \delta_{\zeta,1}(u,v)$$
for every $u,v\in\CVcn$. In particular, we will prove Theorem~\ref{thm:class_iso}. In order to simplify the notation we will write $\MomO=\Mom{1}{n}$ and $\delta_{\zeta}=\delta_{\zeta,1}$.

It is easy to see that every isometry as above is continuous with respect to epi-convergence. Indeed, if $u_k,u\in\CVcn$ are such that $u_k\eto u$, then by Lemma~\ref{le:epi_conv_implies_int_metric_conv}
$$\delta_\zeta(I(u_k),I(u))=\delta_\zeta(u_k,u)\to 0$$
as $k\to\infty$ and thus, by Proposition~\ref{prop:int_metric_conv_implies_epi_conv}, $I(u_k)\eto I(u)$. Similarly, one shows that epi-convergence of $I(u_k)$ to $I(u)$ implies that $u_k$ epi-converges to $u$ as $k\to\infty$.

\medskip

Gruber proved in \cite{gruber_mathematika_1978} that the isometries with respect to the symmetric difference metric on the set of convex bodies with non-empty interior are induced by volume-preserving affine transformations. It is easy to see that the composition with such an affinity defines an isometry on $\CVcn$ with respect to $\delta_\zeta$, $\zeta\in\MomO$. As the following example shows, depending on the function $\zeta$, further isometries, involving general linear transforms, are possible.

\begin{example}
\label{ex:iso_e_c}
For $t\in\R$ let $\zeta(t)=e^{-ct}$ with some $c>0$. Furthermore, for $\phi\in\GLn$ and $x_0\in\Rn$ let $\alpha(x)=\phi(x)+x_0$ for $x\in\Rn$ and let
$$I(u)=u\circ\alpha^{-1} +\frac{\ln(|\det \phi|)}{c}$$
for $u\in\CVcn$. Since 
\begin{align*}
\int_{\Rn} | e^{-cI(u)(x)}-e^{-cI(v)(x)}| \d x &= \int_{\Rn} |e^{-cu(\phi^{-1}(x-x_0))-\ln(|\det \phi|)} - e^{-cv(\phi^{-1}(x-x_0))-\ln(|\det \phi|)}|\d x\\
&=|\det \phi| \int_{\Rn} |e^{-cu(x)-\ln(|\det \phi|)}-e^{-cv(x)-\ln(|\det \phi|)}| \d x\\
&=\int_{\Rn} |e^{-cu(x)}-e^{-cv(x)}|\d x
\end{align*}
for every $u,v\in\CVcn$, it follows that $I$ is an isometry with respect to $\delta_{\zeta}$.
\end{example}

\subsection{Measures}
\label{subse:measures}

Throughout the following, for $\zeta\in\MomO$ we define a positive Radon measure $\nu_\zeta$ on $\R$ via
$$\nu_\zeta([t,\infty))=\zeta(t)$$
for every $t\in\R$. Furthermore, we define a positive Radon measure $\Psi_\zeta$ on $\R^{n+1}$ via
$$\Psi_\zeta(A)=\int_{\Rn} \int_{\R} \chi_A(x,t) \d \nu_\zeta(t) \d x$$
for every Borel set $A\subseteq \R^{n+1}$. It now follows that
\begin{align*}
\Psi_\zeta(\epi u) &= \int_{\Rn} \int_{\R} \chi_{\epi u}(x,t) \d \nu_\zeta(t) \d x\\
&= \int_{\Rn} \int_{\R} \chi_{[u(x),\infty)}(t) \d \nu_\zeta(t) \d x \\
&= \int_{\Rn} \zeta(u(x)) \d x
\end{align*}
and moreover
\begin{align*}
\Psi_\zeta(\epi u \Delta \epi v) &= \int_{\Rn} \int_{\R} \chi_{\epi u \Delta \epi v}(x,t) \d \nu_\zeta \d x\\
&= \int_{\Rn} \int_{\R} \chi_{\big[(u \wedge v)(x),(u \vee v)(x) \big)}(t) \d\nu_\zeta(t) \d x\\
&= \int_{\Rn} \zeta\big((u \wedge v)(x) \big) - \zeta\big((u\vee v)(x) \big) \d x\\
&= \int_{\Rn} |\zeta(u(x))-\zeta(v(x))| \d x\\
&= \delta_\zeta(u,v)
\end{align*}
for every $u,v\in\CVcn$.

\medskip

We conclude this section with the following result, which is easy to see.
\begin{lemma}
\label{le:subsets_same_measure_equality}
Let $\zeta\in\MomO$. If $u,v\in\CVcn$ are such that $\epi u \subseteq \epi v$ and $\Psi_\zeta(\epi u)=\Psi_\zeta(\epi v)$, then $u\equiv v$.
\end{lemma}
\begin{proof}
Let $u,v\in\CVcn$ be as in the statement. We now have $\Psi_\zeta(\epi u \cap \epi v)=\Psi_\zeta(\epi u)=\Psi_\zeta(\epi v)$ and thus
\begin{align*}
\delta_\zeta(u,v)&=\Psi_\zeta(\epi u\Delta \epi v)\\
&= \Psi_\zeta(\epi u) + \Psi_\zeta(\epi v) - 2 \Psi_\zeta(\epi u \cap \epi v)\\
&=0.
\end{align*}
Since $\delta_\zeta$ is a metric (Lemma~\ref{le:delta_zeta_p_is_a_metric}), this implies $u\equiv v$.
\end{proof}

\subsection{Preserving Measure}
It is no coincidence that the isometries in Example~\ref{ex:iso_e_c} are induced by measure preserving maps on $\R^{n+1}$, i.e., $\Psi_\zeta$ of the epigraph of the function is preserved. Indeed, in Proposition~\ref{prop:iso_measure_pres} we will show that every isometry on $\CVcn$ is measure preserving. We will use ideas of Gruber \cite{gruber_mathematika_1978} together with some new arguments, due to the functional setting.

\medskip

The following lemma, which is easy to see, gives sufficient conditions to bound a sequence of functions in $\CVc$ from below.

\begin{lemma}
\label{le:conv_lvl_sets_bnd_min}
Let $t_1,t_2\in\R$ be such that $t_1 < t_2$ and let $K_1,K_2\in\Kn$ be such that $K_1\subsetneq K_2$. If $u_k\in\CVc$ is a sequence such that
$$\{u_k\leq t_1\} \to K_1 \qquad \text{and} \qquad \{u_k\leq t_2\}\to K_2$$
as $k\to \infty$, then there exists $m\in\R$ such that
$$u_k(x)\geq m$$
for every $x\in\Rn$ and $k\in\N$.
\end{lemma}
\begin{proof}
Since $K_1\subsetneq K_2$ we have $d=d_H(K_1,K_2)>0$. By our assumptions on $u_k$ there exists $k_0\in\N$ such that $d_H(\{u_k\leq t_i\},K_i)< \frac d4$, $i\in\{1,2\}$, for every $k\geq k_0$. Hence, it follows from the definition of the Hausdorff distance that there exists $y_k\in \{u_k\leq t_2\}$ such that
\begin{equation}
\label{eq:dist_y_k_x}
\min\nolimits_{x\in \{u_k \leq t_1\}} |y_k - x| \geq \frac d 2
\end{equation}
for every $k\geq k_0$.

For $k\in \N$ let $x_k\in\Rn$ and $m_k\in \R$ be such that $m_k=u_k(x_k)=\min_{x\in\Rn} u_k(x)$. Since for $k\geq k_0$ we have $x_k\in\{u_k\leq t_1\}\subset K_1 + B(0,\frac d4)$ and $y_k\in \{u_k\leq t_2\}\subset K_2 + B(0,\frac d4)$, there exists $R>d$, not depending on $k$, such that
$$|y_k-x_k|<R$$
for every $k\geq k_0$. Fix an arbitrary $\lambda_0>0$ such that $\lambda_0< \frac{d}{2R}<1$ and set $z_k=\lambda_0 x_k+(1-\lambda_0)y_k$. We now have
$$|y_k-z_k|= \lambda_0 |y_k - x_k| < \frac{d}{2}$$
for every $k\geq k_0$ and therefore it follows from  \eqref{eq:dist_y_k_x} that $z_k\not\in \{u_k\leq t_1\}$. Together with the convexity of $u_k$ this shows that
$$t_1 < u_k(z_k) \leq \lambda_0 u_k(x_k) + (1-\lambda_0) u_k(y_k) \leq \lambda_0 m_k + (1-\lambda_0) t_2$$
and therefore
$$\frac{t_1-(1-\lambda_0)t_2}{\lambda_0}< m_k$$
for every $k\geq k_0$. The statement now follows by taking
$$m=\min\left\{m_1,\ldots,m_{k_0-1}, \frac{t_1-(1-\lambda_0)t_2}{\lambda_0}\right\}.$$
\end{proof}

The following two lemmas provide again sufficient criteria to find lower bounds for (sub)sequences of functions, depending on the behavior of $\zeta\in\MomO$.

\begin{lemma}
\label{le:subseq_unif_min_case_1}
Let $\zeta\in\MomO$ be such that $\lim_{t\to\infty}\zeta(-t)=\infty$ and let $u_k\in\CVcn$ be a sequence. If there exist $K_1\in\Knn$, $C>0$ and $t_1\in\R$ such that $\Psi_\zeta(u_k)\leq C$ for every $k\in\N$ and
\begin{equation}
\label{eq:conv_lvl_sets_k_1}
\{u_k \leq t_1\}\to K_1
\end{equation}
as $k\to\infty$, then there exist a subsequence $u_{k_j}$ of $u_k$ and $m\in\R$ such that
$$u_{k_j}\geq m$$
for every $x\in\Rn$ and $j\in\N$.
\end{lemma}
\begin{proof}
By our assumptions on $\zeta$ we can find $t_0<t_1$ such that
\begin{equation}
\label{eq:choice_t_0}
\frac{2\,C}{V_n(K_1)} < \zeta(t_0).
\end{equation}
In case $\{u_k \leq t_0\}=\emptyset$ for almost every $k\in\N$ there is nothing to show. Therefore assume that this is not the case. By \eqref{eq:conv_lvl_sets_k_1} the sets $\{u_k\leq t_1\}$ are uniformly bounded and therefore also the sets $\{u_k \leq t_0\}$ have this property. Thus, it follows from the Blaschke selection theorem (Theorem~\ref{thm:blaschke_selection_cvx_bds}) that there exist a subsequence $u_{k_j}$ of $u_k$ and $K_0\in\Kn$ such that
\begin{equation}
\label{eq:conv_sub_lvl_sets_k_0}
\{u_{k_j}\leq t_0\}\to K_0\subseteq K_1
\end{equation}
as $j\to\infty$. By our assumptions on the sequence $u_k$ and \eqref{eq:choice_t_0} we now have
$$V_n(\{u_{k_j}\leq t_0\}) \zeta(t_0)\leq \Psi_\zeta(\epi u_{k_j}) \leq C < \frac{V_n(K_1)}{2}\zeta(t_0)$$
for every $j\in\N$. Hence, it follows from \eqref{eq:conv_sub_lvl_sets_k_0} that $V_n(K_0) \leq \frac{V_n(K_2)}{2}$ and thus $K_0\subsetneq K_1$. The statement now follows from Lemma~\ref{le:conv_lvl_sets_bnd_min}.
\end{proof}

\begin{lemma}
\label{le:subseq_unif_min_case_2}
Let $\zeta\in\MomO$ be such that $\lim_{t\to\infty}\zeta(-t)=A\in\R$ and let $u_k\in\CVcn$ be a sequence. If there exist $a,\varepsilon,R>0$ and $b\in\R$ such that
\begin{equation}
\label{eq:ass_u_k_cone_R}
u_k(x)\geq a |x|+b
\end{equation}
for every $k\in\N$ and $x\in\Rn$ with $|x|\geq R$ and
\begin{equation}
\label{eq:ass_pairwise_distance}
\delta_\zeta(u_k,u_l)=2\varepsilon
\end{equation}
for every $k\neq l$, then there exist a subsequence $u_{k_j}$ of $u_k$ and $m\in\R$ such that
$$u_{k_j}\geq m$$
for every $x\in\Rn$ and $j\in\N$.
\end{lemma}
\begin{proof}
By \eqref{eq:ass_u_k_cone_R} there exists $t_1\in\R$ such that
$$\Psi_\zeta(\epi u_k\cap \{x_{n+1}\geq t_1\}) < \frac{\varepsilon}{2}$$
for every $k\in\N$. In particular, we now have
\begin{multline}
\label{eq:psi_zeta_t_1}
\Psi_\zeta((\epi u_k \Delta \epi u_l) \cap \{x_{n+1} \geq t_1\})\\
\leq \Psi_\zeta(\epi u_k \cap \{x_{n+1} \geq t_1\}) + \Psi_\zeta(\epi u_l \cap \{x_{n+1} \geq t_1\}) < \varepsilon
\end{multline}
for every $k,l\in\N$. Observe that \eqref{eq:ass_u_k_cone_R} also implies that the sets $\{u_k\leq t_1\}$ are uniformly bounded and assume that $\{u_k\leq t_1\}\neq \emptyset$ for almost every $k\in\N$ since otherwise the statement is trivial. By the Blaschke selection theorem (Theorem~\ref{thm:blaschke_selection_cvx_bds}) there exist a subsequence $u_{k_j}$ of $u_k$ and $K_1\in\Kn$ such that
$$\{u_{k_j}\leq t_1\}\to K_1$$
as $j\to \infty$. We claim that $K_1$ must have non-empty interior. Otherwise can find $j_0\in\N$ such that
$$V_n(\{u_{k_j}\leq t_1\})< \frac{\varepsilon}{2 (A-\zeta(t_1))}$$
for every $j\geq j_0$ and therefore
\begin{align*}
\Psi_\zeta((\epi u_{k_j} \Delta \epi u_{k_l}) \cap \{x_{n+1}\leq t_1\})&\leq \Psi_\zeta(\{u_{k_j}\leq t_1\} \times (-\infty,t_1]) + \Psi_\zeta(\{u_{k_j}\leq t_1\} \times (-\infty,t_1])\\
&= V_n(\{u_{k_j}\leq t_1\})(A-\zeta(t_1)) + V_n(\{u_{k_j}\leq t_1\})(A-\zeta(t_1))\\
&<\varepsilon
\end{align*}
for every $j,l\geq j_0$. Together with \eqref{eq:psi_zeta_t_1} we now have
$$
\delta_\zeta(u_{k_j},u_{k_l}) = \Psi_\zeta((\epi u_{k_j} \Delta \epi u_{k_l}) \cap \{x_{n+1}\leq t_1\}) + \Psi_\zeta((\epi u_{k_j} \Delta \epi u_{k_l}) \cap \{x_{n+1}\geq t_1\}) < 2 \varepsilon
$$
for every $j,l\geq j_0$ which is a contradiction by \eqref{eq:ass_pairwise_distance}. Thus, $V_n(K_1)>0$.

By our assumptions on $\zeta$ we can find $t_0<t_1$ such that
$$V_n(K_1+B(0,\varepsilon))(A-\zeta(t_0))<\frac{\varepsilon}{2}.$$
Again, we will assume that $\{u_{k_j}\leq t_0\}\neq \emptyset$ for almost every $j\in\N$ since otherwise the statement is trivial. Hence, after possibly restricting to another subsequence, there exists $K_0\in\Kn$ such that
$$\{u_{k_j}\leq t_0\}\to K_0\subseteq K_1$$
as $j\to\infty$. Assume that $K_0=K_1$. Since $V_n(K_0)=V_n(K_1)>0$ we can find $j_0\in\N$ and $L_0,L_1\in\Kn$ such that
$$d_H(\{u_{k_j}\leq t_0\},K_0)<\varepsilon,\quad L_0\subseteq \{u_{k_j}\leq t_0\},\quad \{u_{k_j}\leq t_1\}\subseteq L_1,\quad V_n(L_1\backslash L_0) < \frac{\varepsilon}{2(A-\zeta(t_1))}$$
for every $j\geq j_0$. Since $L_0 \times [t_0,t_1]\subset \epi u_{k_j}$ for every $j\geq j_0$, this shows
\begin{align*}\Psi_\zeta((\epi u_{k_j} \Delta \epi u_{k_l}) \cap \{x_{n+1}\leq t_1\})&\leq \Psi_\zeta((K_0+B(0,\varepsilon))\times (-\infty,t_0]) + \Psi_\zeta((L_1\backslash L_0) \times [t_0,t_1])\\
&= V_n(K_0+B(0,\varepsilon))(A-\zeta(t_0)) + V_n(L_1\backslash L_0)(\zeta(t_0)-\zeta(t_1))\\
&< \varepsilon
\end{align*}
for every $j,l\geq j_0$. Similarly as before, together with \eqref{eq:psi_zeta_t_1} this shows that $\delta_\zeta(u_{k_j},u_{k_l})< 2 \varepsilon$ for every $j,l\geq j_0$ which contradicts \eqref{eq:ass_pairwise_distance}. Hence, we must have $K_0\subsetneq K_1$ and the statement now follows from Lemma~\ref{le:conv_lvl_sets_bnd_min}.
\end{proof}

We can now formulate necessary conditions for sequences of functions in $\CVcn$ which allow us to find subsequences that are bounded in the sense of Theorem~\ref{thm:blaschke_selection_cvx_fcts}.

\begin{lemma}
\label{le:cond_seq_unif_bd}
Let $\zeta\in\MomO$, $u\in\CVcn$ and $C,\varepsilon>0$. If $v_k\in\CVcn$ is a sequence such that
$$\Psi_\zeta(\epi u\cap \epi v_k)\geq \varepsilon \qquad \text{and}\qquad \delta_\zeta(u,v_k) \leq C$$
for every $k\in\N$ and
\begin{equation}
\label{eq:cond_delta_zeta_2_eps}
\delta_{\zeta}(v_k,v_l)=2\varepsilon
\end{equation}
for every $k\neq l$, then there exist a subsequence $v_{k_j}$ of $v_k$, $a>0$ and $b,t_0\in\R$ such that $\min_{x\in\Rn} v_{k_j}(x) \leq t_0$ and
\begin{equation}
\label{eq:v_k_j_unif_cone}
v_{k_j}(x)\geq a|x|+b
\end{equation}
for every $x\in\Rn$ and $j\in\N$.
\end{lemma}
\begin{proof}
Let $m:=\min\nolimits_{x\in\Rn} u(x)$. By the properties of $\zeta$ there exists $t_0\in \R$ such that for every $t\geq t_0$
$$\Psi_\zeta(\epi u \cap \{x_{n+1}\geq t\}) \leq \frac{\varepsilon}{2}$$
and therefore
\begin{align*}
\varepsilon & \leq \Psi_\zeta(\epi u \cap \epi v_k)\\
&= \Psi_\zeta(\epi u \cap \epi v_k \cap \{x_{n+1} \geq t\}) + \Psi_\zeta(\epi u \cap \epi v_k \cap \{x_{n+1}\leq t\})\\
&\leq \Psi_\zeta(\epi u \cap \{x_{n+1} \geq t\}) + \Psi_\zeta((\{u\leq t\}\cap\{v_k \leq t\}) \times [m, t])\\
&\leq \frac{\varepsilon}{2} + V_n(\{u\leq t\}\cap\{v_k \leq t\}) (\zeta(m)-\zeta(t))
\end{align*}
for every $k\in\N$. Hence, for every $t\geq t_0$ there exists a positive lower bound for\linebreak$V_n(\{u\leq t\}\cap \{v_k\leq t\})$, independent of $k\in\N$, and thus there exists $r(t)>0$ such that each of the sets $\{u\leq t\}\cap \{v_k\leq t\}$ contains a ball of radius $r(t)$. Furthermore this also shows that
$$\min\nolimits_{x\in\Rn} v_k(x)\leq t_0.$$
Since
\begin{equation}
\label{eq:psi_zeta_v_k_bd}
\Psi_\zeta(\epi v_k) \leq \delta_\zeta(u,v_k) + \Psi_\zeta(\epi u) \leq C + \Psi_\zeta(\epi u)
\end{equation}
and
$$\Psi_\zeta(\epi v_k) \geq \Psi_\zeta(\{v_k\leq t\} \times [t,\infty)) = V_n(\{v_k \leq t\}) \zeta(t)$$
it follows that
$$V_n(\{v_k\leq t\}) \leq \frac{C+\Psi_\zeta(\epi u)}{\zeta(t)}$$
for every $t\geq t_0$ and $k\in\N$. In particular, this implies that there exists $R(t)>0$ such that
\begin{equation}
\label{eq:uniform_balls_lvl_sets}
\{v_k\leq t\}\subseteq B(0,R(t))
\end{equation}
for every $k\in\N$, since $\{v_k\leq t\}$ contains a ball of radius $r(t)$ that is contained in $\{u\leq t\}$.

\medskip

Fix an arbitrary $t_1>t_0$. Since $V_n(\{v_k\leq t_i\})\subset B(0,R(t_i))$, $i\in\{0,1\}$, for every $k\in\N$ it follows from the Blaschke selection theorem (Theorem~\ref{thm:blaschke_selection_cvx_bds}) that there exist $K_0,K_1\in \Kn$ and a subsequence $v_{k_j}$ of $v_k$ such that 
\begin{equation}
\label{eq:conv_v_k_j_lvl_sets}
\{v_{k_j} \leq t_i\} \to K_i
\end{equation}
for $i\in\{0,1\}$ as $j\to\infty$. In particular, since there are uniform positive lower bounds for $V_n(\{v_k\leq t_0\})$ as well as $V_n(\{v_k\leq t_1\})$, both $K_0$ and $K_1$ have non-empty interiors. Furthermore, since $\{v_k\leq t_0\}\subseteq \{v_k\leq t_1\}$ for every $k\in\N$ we must have $K_0\subseteq K_1$. Moreover, there exist $x_0\in\Rn$ and $j_0\in\N$ such that $x_0\in \interior\{v_{k_j}\leq t_0\}$. Without loss of generality we will assume $x_0=0$ and therefore
$$v_{k_j}(0)\leq t_0$$
for every $j\geq j_0$. On the other hand, by \eqref{eq:uniform_balls_lvl_sets},
$$v_{k_j}(R(t_1)\, e) \geq t_1$$
for every $e\in\Sph$ and $j\geq j_0$. Hence, by convexity
$$\frac{v_{k_j}(r\,e)-v_{k_j}(R(t_1)\,e) }{r-R(t_1)} \geq \frac{v_{k_j}(R(t_1)\, e) - v_{k_j}(0)}{|R(t_1)\,e-0|}\geq \frac{t_1-t_0}{R(t_1)}$$
for every $r\geq R(t_1)$ and $e\in\Sph$ and thus
\begin{equation}
\label{eq:v_k_j_cone_r_t_1}
v_{k_j}(x) \geq \frac{t_1-t_0}{R(t_1)} |x| +t_0
\end{equation}
for every $x\in\Rn$ with $|x|\geq R(t_1)$ and $j\geq j_0$.

It remains to show that, after possibly restricting to another subsequence, there exists $b\in\R$ such that $v_{k_j}(x)\geq b$ for every $x\in\Rn$ and $j\geq j_0$. It is easy to see that together with \eqref{eq:v_k_j_cone_r_t_1} this then implies \eqref{eq:v_k_j_unif_cone}. If $\lim_{t\to\infty}\zeta(-t)=\infty$ this follows from \eqref{eq:psi_zeta_v_k_bd} and \eqref{eq:conv_v_k_j_lvl_sets} together with Lemma~\ref{le:subseq_unif_min_case_1}. In the remaining case $\lim_{t\to\infty}\zeta(-t)=A\in\R$ this follows from \eqref{eq:cond_delta_zeta_2_eps}, \eqref{eq:v_k_j_cone_r_t_1} and Lemma~\ref{le:subseq_unif_min_case_2}.
\end{proof}

\begin{proposition}
\label{prop:iso_measure_pres}
Let $\zeta\in \MomO$. If $I:(\CVcn,\delta_{\zeta})\to(\CVcn,\delta_{\zeta})$ is an isometry, then
$$\int_{\Rn} \zeta(I(u)(x)) \d x = \int_{\Rn} \zeta(u(x)) \d x$$
for every $u\in\CVcn$.
\end{proposition}
\begin{proof}
Throughout the proof fix an arbitrary $u\in\CVcn$. Suppose first that $\Psi_\zeta(\epi u) < \Psi_\zeta(\epi I(u))$ and let $\varepsilon>0$ be such that
\begin{equation}
\label{eq:choice_vareps}
\Psi_\zeta(\epi u)<\Psi_\zeta(\epi I(u))-2\varepsilon.
\end{equation}
Let $K_i\in\Kn$ be a sequence of pairwise disjoint convex bodies such that $V_n(K_i)=\varepsilon/\zeta(0)$ for $i\in\N$ and let $v_i=\Ind_{K_i}$. We now have
$$\Psi_\zeta(\epi v_i) = V_n(K_i)\zeta(0)=\varepsilon$$
and furthermore
\begin{equation}
\label{eq:delta_ind_k_i_2_eps}
\delta_\zeta(I(v_i),I(v_j))=\delta_\zeta(v_i,v_j) = 2\varepsilon
\end{equation}
for every $i\neq j$. Moreover, by \eqref{eq:choice_vareps},
$$\delta_\zeta(I(u),I(v_i))=\delta_\zeta(u,v_i)\leq \Psi_\zeta(\epi u) + \Psi_\zeta(\epi v_i) \leq \Psi_\zeta(\epi I(u))-\varepsilon$$
for every $i\in \N$. Furthermore,
\begin{align*}
\Psi_\zeta(\epi I(u) \cap \epi I(v_i)) &= \Psi_\zeta(\epi I(u)) - \Psi_\zeta(\epi I(u)\backslash \epi I(v_i))\\
&\geq \Psi_\zeta(\epi I(u)) - \Psi_\zeta(\epi I(u) \Delta \epi I(v_i))\\
&= \Psi_\zeta(\epi I(u)) - \delta_{\zeta}(I(u),I(v_i))\\
&\geq \Psi_\zeta(\epi I(u)) - (\Psi_\zeta(\epi I(u))-\varepsilon)\\
&= \varepsilon.
\end{align*}
for every $i\in\N$. Thus, by Lemma~\ref{le:cond_seq_unif_bd} there exist a subsequence $I(v_{i_j})$ of $I(v_i)$, $a>0$ and $b,m\in\R$ such that
$$\min\nolimits_{x\in\Rn} I(v_{i_j})(x)\leq m$$
and
$$I(v_{i_j})(x)\geq a|x|+b$$
for every $x\in\Rn$ and $j\in\N$. Hence, by Lemma~\ref{thm:blaschke_selection_cvx_fcts} there exists an epi-convergent subsequence $w_k:=I(v_{i_{j_k}})$ of $I(v_{i_j})$ with limit function $w\in\CVc$. In case $w\in\CVcn$, it follows from Lemma~\ref{le:epi_conv_implies_int_metric_conv} that there exist $k_1,k_2\in\N$ such that $\delta_{\zeta}(w_{k_1},w)< \varepsilon$ and $\delta_{\zeta}(w_{k_2},w)<\varepsilon$. In case $w\not\in\CVcn$, it follows from Lemma~\ref{le:int_continuous} that $\Psi_\zeta(w_k)\to \Psi_\zeta(w)=0$. Thus, we conclude that in both cases there exist $k_1,k_2\in\N$ such that
$$\delta_\zeta(w_{k_1},w_{k_2}) < 2 \varepsilon$$
which contradicts \eqref{eq:delta_ind_k_i_2_eps}. Hence, $\Psi_\zeta(\epi u ) \geq \Psi_\zeta(\epi I(u))$ for every $u\in\CVcn$.

Next, let $R>0$ and denote $u_R=u+\Ind_{B(0,R)}$. Suppose that $\Psi_\zeta(\epi u_R) > \Psi_\zeta(\epi I(u_R) )$. For any $v\in\CVcn$ such that $\dom v \cap B(0,R)=\emptyset$ we have
\begin{align*}
\Psi_\zeta(\epi u_R)+\Psi_\zeta(\epi v) &= \delta_\zeta(u_R,v)\\
&= \delta_\zeta(I(u_R),I(v))\\
&\leq \Psi_\zeta(\epi I(u_R)) + \Psi_\zeta(\epi I(v))\\
&< \Psi_\zeta(\epi u_R) + \Psi_\zeta(\epi I(v))
\end{align*}
and therefore $\Psi_\zeta(\epi v) < \Psi_\zeta(\epi I(v))$ which is a contradiction by the first part of the proof. Hence, $\Psi_\zeta(\epi u_R) \leq \Psi_\zeta(\epi I(u_R))$ for every $R>0$. Together with the fact that $u_R\eto u$ and therefore also $I(u_R)\eto I(u)$ as $R\to \infty$ it now follows from Lemma~\ref{le:int_continuous} that also $\Psi_\zeta(\epi u)\leq \Psi_\zeta(\epi I(u))$.
\end{proof}

\begin{remark}
Note that the reverse statement of Proposition~\ref{prop:iso_measure_pres} is not true. For example, fix $x_0\in\Rn\backslash\{0\}$ and define $I:(\CVcn,\delta_\zeta)\to(\CVcn,\delta_\zeta)$ as
$$I(u)(x)=u(x-\Psi_\zeta(u)x_0)$$
for every $u\in\CVcn$ and $x\in\Rn$. Since $I(u)$ is just a translation of $u$, it is easy to see that $\Psi_\zeta(u)=\Psi_\zeta(I(u))$ for every $u\in\CVcn$. However, since this translation depends on the function itself, it follows that in general $\delta_\zeta(u,v)\neq \delta_\zeta(I(u),I(v))$ for $u,v\in\CVcn$.
\end{remark}

We will need the following easy consequences of the last result.

\begin{lemma}
\label{le:isometry_preserves_order_cup_cap}
Let $\zeta\in\MomO$ and let $I:(\CVcn,\delta_{\zeta})\to(\CVcn,\delta_{\zeta})$ be an isometry. The following hold true:
\begin{enumerate}[label=(\alph*)]
    \item\label{it:measure_cap} For every $u,v\in\CVcn$ we have
    $$\Psi_\zeta(\epi I(u) \cap \epi I(v))=\Psi_\zeta(\epi u \cap \epi v).$$
    \item\label{it:strict_incl} If $u,v\in\CVcn$ are such that $\epi u \subsetneq \epi v$, then
    $$\epi I(u) \subsetneq \epi I(v).$$
     \item\label{it:pres_vee} If $u,v\in\CVcn$ are such that $u\vee v \in \CVcn$, then
     $$I(u \vee v)= I(u) \vee I(v).$$
    \item\label{it:pres_wedge} If $u,v\in\CVcn$ are such that $u\wedge v \in\CVcn$, then
    $$I(u \wedge v) = I(u) \wedge I(v).$$
    In particular, $I(u)\wedge I(v)$ is convex.
\end{enumerate}
\end{lemma}
\begin{proof}
We will prove each of the statements separately.
\begin{enumerate}[label=(\alph*)]
    \item Since $I$ is an isometry we have
    \begin{align*}
    \Psi_\zeta(\epi I(u))+\Psi_\zeta(\epi I(v))&-2 \Psi_\zeta(\epi I(u)\cap \epi I(v))\\
    &= \Psi_\zeta(\epi I(u) \Delta \epi I(v))\\
    &= \delta_\zeta(I(u),I(v))\\
    &= \delta_\zeta(u,v)\\
    &= \Psi_\zeta(\epi u)+ \Psi_\zeta(\epi v)- 2\Psi_\zeta(\epi u \cap \epi v)
    \end{align*}
    and therefore, by Proposition~\ref{prop:iso_measure_pres},
    \begin{equation*}
    \Psi_\zeta(\epi I(u) \cap \epi I(v)) = \Psi_\zeta(\epi u \cap \epi v)
    \end{equation*}
    for every $u,v\in\CVcn$.

    \item Let $u,v\in\CVcn$ be such that $\epi u \subsetneq \epi v$. By \ref{it:measure_cap}, Proposition~\ref{prop:iso_measure_pres} and our assumptions on $u$ and $v$, we have
    \begin{align*}
        \Psi_\zeta(\epi u \cap \epi v) &= \Psi_\zeta(\epi I(u)\cap \epi I(v))\\
        &\leq \Psi_\zeta(\epi I(u))\\
        &= \Psi_\zeta(\epi u)\\
        &=\Psi_\zeta(\epi u \cap \epi v)
    \end{align*}
    and therefore $\Psi_\zeta(\epi I(u)\cap \epi I(v)) = \Psi_\zeta(\epi I(u))$. Hence, by Lemma~\ref{le:subsets_same_measure_equality},
    $$\epi I(u)=\epi I(u)\cap \epi I(v)\subseteq \epi I(v).$$
    Since $\Psi_\zeta(\epi u)<\Psi_\zeta(\epi v)$ it follows from Proposition~\ref{prop:iso_measure_pres} that also $\Psi_\zeta(\epi I(u))<\Psi_\zeta(\epi I(v))$ and thus, $\epi I(u)\subsetneq \epi I(v)$.
    
    \item Let $u,v \in\CVcn$ be such that also $u \vee v \in\CVcn$. It follows from Proposition~\ref{prop:iso_measure_pres} and \ref{it:measure_cap} that
    \begin{equation}
    \label{eq:I_interch_cap_preserves_measure}
    \Psi_\zeta(\epi I(u \vee v)) = \Psi_\zeta(\epi (u\vee v)) = \Psi_\zeta(\epi I(u) \cap \epi I(v)).
    \end{equation}
    Furthermore, since $\epi(u\vee v)=\epi u \cap \epi v \subseteq \epi u$ we obtain from \ref{it:strict_incl} that $\epi I(u\vee v) \subseteq \epi I(u)$ and similarly $\epi I(u\vee v)\subseteq \epi I(v)$. Thus,
    $$\epi I(u\vee v) \subseteq \epi I(u) \cap \epi I(v).$$
    By \eqref{eq:I_interch_cap_preserves_measure} and Lemma~\ref{le:subsets_same_measure_equality} this can only be the case if $\epi I(u\vee v) = \epi I(u) \cap \epi I(v)$.
    
    \item Let $u,v\in\CVcn$ be such that $u \wedge v \in \CVcn$. Since
    $$\epi u \subseteq \epi u \cup \epi v =\epi(u \wedge v)$$
    and similarly, $\epi v \subseteq \epi (u \wedge v)$ it follows from \ref{it:strict_incl} that
    \begin{equation}
    \label{eq:epi_wedge_i_eq_epi_i_wedge}
    \epi(I(u)\wedge I(v)) = \epi I(u) \cup \epi I(v) \subseteq \epi I(u\wedge v).
    \end{equation}
    Observe that either $u\vee v \in \CVcn$ or $\Psi_\zeta(\epi u \cap \epi v)=0$.
    Hence, it follows from Proposition~\ref{prop:iso_measure_pres} and \ref{it:measure_cap} that
    \begin{align*}
    \Psi_\zeta(\epi I(u\wedge v)) &= \Psi_\zeta(\epi (u \wedge v))\\
    &=\Psi_\zeta(\epi u) + \Psi_\zeta(\epi v) - \Psi_\zeta(\epi u \cap \epi v)\\
    &=\Psi_\zeta(\epi I(u))+ \Psi_\zeta(\epi I(v)) - \Psi_\zeta(\epi I(u) \cap \epi I(v))\\
    &=\Psi_\zeta(\epi(I(u)\wedge I(v))).
    \end{align*}
    It now follows from \eqref{eq:epi_wedge_i_eq_epi_i_wedge} and Lemma~\ref{le:subsets_same_measure_equality} that $\epi(I(u)\wedge I(v)) = \epi I(u\wedge v)$.\qedhere
\end{enumerate}
\end{proof}

\subsection{Parallelotopes}
A convex body $P\in\Kn$ is called a \emph{parallelotope} if it can be written as the Minkowski sum of $n$ segments. We will always assume that parallelotopes are not degenerated or equivalently $V_n(P)>0$. The aim of this section is to prove Proposition~\ref{prop:isometries_parallelotopes} which states that isometries map indicator functions of parallelotopes to functions whose domain is again a parallelotope. As before, we will follow the ideas of Gruber \cite{gruber_mathematika_1978}.

\medskip

Throughout the following we will say that two parallelotopes $P,Q\in\Kn$ are \emph{neighbors} if $Q=P+b$, where $b\in\Rn$ is a vector that is parallel to a one-dimensional edge of $P$ and of the same length. It can be easily seen that this equivalent to the fact that $V_n(P)=V_n(Q)$, $P\cup Q$ is convex and $P\cap Q$ is a common facet of $P$ and $Q$. In particular, $V_n(P\cap Q)=0$.

\begin{lemma}
\label{le:parallelotope_2n_facets}
Let $\zeta\in \MomO$. If $I:(\CVcn,\delta_\zeta)\to(\CVcn,\delta_\zeta)$ is an isometry, then for every parallelotope $P\in\Kn$ and $t\in\R$ the set $\dom I(\Ind_P+t)$ has $2n$ different facets of the form
$$\dom I(\Ind_P+t)\cap \dom I(\Ind_Q+t)$$
where $Q\in\Kn$ is a neighbor of $P$.
\end{lemma}
\begin{proof}
Throughout the proof fix a parallelotope $P$ and $t\in\R$. Let $Q$ be a neighbor of $P$. Since $\Psi_\zeta(\epi (\Ind_P+t) \cap \epi (\Ind_Q+t))=0$ it follows from Lemma~\ref{le:isometry_preserves_order_cup_cap}~\ref{it:measure_cap} that
$$\Psi_{\zeta}(\epi I(\Ind_P+t) \cap \epi I(\Ind_Q+t))=0$$
and since $P\cup Q$ is convex it follows from \ref{it:pres_wedge} of the same lemma that $I(\Ind_P+t)\wedge I(\Ind_Q+t)$ is convex, as well.
Therefore,
$$V_n(\dom I(\Ind_P+t) \cap \dom I(\Ind_Q+t))=0$$
and $\dom I(\Ind_P+t) \cup \dom I(\Ind_Q+t)$ is convex. Hence, it is easy to see that
$$\dom I(\Ind_P+t) \cap \dom I(\Ind_Q+t)$$
is a facet of $\dom I(\Ind_P+t)$.

Next, suppose that $Q_1$ and $Q_2$ are two different neighbors of $P$ such that
$$\dom I(\Ind_P+t) \cap \dom I(\Ind_{Q_1}+t) = \dom I(\Ind_P+t) \cap \dom I(\Ind_{Q_2}+t).$$
Since by Proposition~\ref{prop:iso_measure_pres} the convex sets $\dom I(\Ind_{Q_1}+t)$ and $\dom I(\Ind_{Q_2}+t)$ are full-dimensional, the first part of the proof shows that they must have common interior points. On the other hand, since $V_n(Q_1 \cap Q_2)=0$ it follows from Lemma~\ref{le:isometry_preserves_order_cup_cap}~\ref{it:measure_cap}, similar as above, that
$$V_n(\dom I(\Ind_{Q_1}+t) \cap \dom I(\Ind_{Q_2}+t))=0$$
which is a contradiction.

We conclude that each of the $2n$ neighbors of $P$ corresponds to a different facet of\linebreak${\dom(\Ind_P+t)}$, which completes the proof. 
\end{proof}

Let $R$ be a ray in $\Rn$ starting from the origin and let $x\in\bd P$ be a boundary point of $P\in\Kn$. We say that one can \emph{see infinity} from $x$ via $R$ if $(x+R)\cap \interior P = \emptyset$.

We will need the following characterization by Buchman and Valentine.
\begin{lemma}[\!\!\cite{buchman_valentine}, Theorem 1]
\label{le:buchma_valentine}
A set $P\in\Kn$ is a parallelotope if and only if for each set of $2n-1$ boundary points of $P$ there exists a ray $R$ starting from the origin such that from each of these $2n-1$ points one can see infinity via $R$.
\end{lemma}

\begin{lemma}
Let $\zeta\in \MomO$, let $I:(\CVcn,\delta_\zeta)\to(\CVcn,\delta_\zeta)$ be an isometry, let $t\in\R$ and let $P\in\Kn$ be a parallelotope. If $p_1,\ldots,p_{2n-1}\in\bd \dom I(\Ind_P+t)$ are arbitrary boundary points of $\dom I(\Ind_P+t)$, then there exists a ray $R$ starting from the origin such that one can see infinity via this ray from each of the points $p_1,\ldots,p_{2n-1}$.
\end{lemma}
\begin{proof}
Throughout the proof fix $p_1,\ldots,p_{2n-1}\in\bd \dom I(\Ind_P+t)$. By Lemma~\ref{le:parallelotope_2n_facets} there exist at least one neighbor $P_1$ of $P$ and a facet of $\dom I(\Ind_P+t)$ of the form
$$\dom I(\Ind_P+t) \cap \dom I(\Ind_{P_1}+t)$$
which does not contain any of the points $p_1,\ldots,p_{2n-1}$ in its relative interior. Since $P_1$ is a neighbor of $P$ there exists $b\in\Rn$ such that $P_1=P+b$. Now, set $Q_1=P_1$ and for $i\in\N$, $i\geq 2$, set $P_i=P+ib$ and
$$Q_i=\bigcup_{j=1}^i P_i=(P+b)\cup \cdots \cup (P+ib).$$
Note, that each of the sets $Q_i$ is a parallelotope itself and that $Q_i\subsetneq Q_{i+1}$. Furthermore, $P\cap Q_i$ is a facet of $P$, $V_n(P\cap Q_i)=0$ and $P\cup Q_i$ is convex for every $i\in\N$. Moreover, $\lim_{i\to\infty} V_n(Q_i)=+\infty$. Thus, it follows from Lemma~\ref{le:isometry_preserves_order_cup_cap} \ref{it:strict_incl} that
the sequence $\epi I(\Ind_{Q_i}+t)$ is strictly increasing and by Proposition~\ref{prop:iso_measure_pres} $\lim_{i\to\infty}\Psi_\zeta(\epi I(\Ind_{Q_i}+t))=+\infty$. Furthermore, by Lemma~\ref{le:isometry_preserves_order_cup_cap} \ref{it:measure_cap},
$$\Psi_\zeta(\epi I(\Ind_{P}+t) \cap \epi I(\Ind_{Q_i}+t))=0$$
and therefore also $V_n(\dom I(\Ind_P+t) \cap \dom I(\Ind_{Q_i}+t))=0$ for every $i\in\N$. Moreover, by \ref{it:pres_wedge} of the same lemma $\dom I(\Ind_P+t)\cup\dom I(\Ind_{Q_i}+t)$ is convex and therefore
$$\dom I(\Ind_P+t)\cap\dom I(\Ind_{P_1}+t) = \dom I(\Ind_P+t)\cap \dom I(\Ind_{Q_i}+t)$$
for every $i\in\N$. Now, let 
$$C=\bigcup_{i=1}^\infty \dom I(\Ind_{Q_i}+t) = \bigcup_{i=1}^\infty \dom I(\Ind_{P_i}+t).$$
We want to show that $C$ is unbounded. Assume on the contrary that $C$ is bounded. Since the sets $\dom I(\Ind_{P_i}+t)$ have pairwise disjoint interiors, we have
$$\sum_{i=1}^\infty V_n(\dom I(\Ind_{P_i}+t))=V_n(C)<\infty$$
and therefore $\lim_{i\to\infty} V_n(\dom I(\Ind_{P_i}+t))=0$. By Proposition~\ref{prop:iso_measure_pres} and the definition of $\Psi_\zeta$ we have
$$\text{const.} = \Psi_\zeta(I(\Ind_{P_i}+t)) \leq \zeta(\min\nolimits_{x\in \Rn} I(\Ind_{P_i}+t)(x)) V_n(\dom I(\Ind_{P_i}+t)),$$
for every $i\in\N$. Hence, it must follow that
$$
\lim\nolimits_{i\to\infty} \zeta(\min\nolimits_{x\in \Rn} I(\Ind_{P_i}+t)(x)) = +\infty.
$$
In case $\lim_{t\to\infty}\zeta(-t)=A<\infty$ this is a contradiction. In the remaining case $\lim_{t\to\infty}\zeta(-t)=\infty$ it follows from the properties of $\zeta$ that
\begin{equation}
\label{eq:i_ind_p_i_-inf}
\lim\nolimits_{i\to\infty} \min\nolimits_{x\in\Rn} I(\Ind_{P_i}+t)(x) = -\infty.
\end{equation}
Since $V_n(\dom I(\Ind_{P_1}+t))>0$ there exist $x_0\in\Rn$ and $r>0$ such that
$$B(x_0,r)\subseteq \dom I(\Ind_{P_1}+t).$$
Furthermore, since we assumed that $C$ is bounded there exists $d>0$ such that $C\subseteq B(x_0,d\cdot r)$.
By Lemma~\ref{le:isometry_preserves_order_cup_cap} \ref{it:pres_wedge} the function
$$u_i:=I(\Ind_{P_1}+t)\wedge \cdots \wedge I(\Ind_{P_i}+t)$$
is convex for every $i\in\N$. Now, for every $x\in \dom I(\Ind_{P_i}+t)$ with $i\geq 2$ there exist $y\in \bd B(x_0,r)$ and $1<\lambda \leq d$ such that
$$x=x_0+\lambda(y-x_0)=\lambda y + (1-\lambda)x_0$$
and therefore, by convexity,
\begin{align*}
I(\Ind_{P_i}+t)(x)&=u_i(\lambda y + (1-\lambda)x_0)\\
&\geq \lambda u_i(y)+(1-\lambda)u_i(x_0)\\
&= \lambda I(\Ind_{P_1}+t)(y)+ (1-\lambda) I(\Ind_{P_1}+t)(x_0)\\
&\geq \lambda \min\nolimits_{z\in \bd B(x_0,r)}I(\Ind_{P_1}+t)(z)+ (1-\lambda) I(\Ind_{P_1}+t)(x_0)
\end{align*}
which contradicts \eqref{eq:i_ind_p_i_-inf}. Thus, $C$ must be unbounded.

\medskip
Now let $K$ be the closed convex hull of $C$. By the properties of $C$ the set $K$ is closed, unbounded and convex such that $K\cap \dom I(\Ind_P+t) = \dom I(\Ind_P+t) \cap \dom I(\Ind_{P_1}+t)$ and $K\cup \dom I(\Ind_P+t)$ is convex. Thus, by the choice of $P_1$, the points $p_1,\ldots,p_{2n-1}$ are also boundary points of $K\cup \dom I(\Ind_P+t)$. Furthermore, by convexity, there exists a ray $R$ starting from the origin such that $p+R\subset K \cup \dom I(\Ind_P+t)$ for every $p\in K\cup \dom I(\Ind_P+t)$. In particular, $$p_i+R \subset K\cup \dom I(\Ind_P+t)$$
for every $i\in\{1,\ldots,2n-1\}$. Since $p_1,\ldots,p_{2n-1}$ are boundary points of $K\cup \dom I(\Ind_P+t)$ the rays $p_1-R,\ldots,p_{2n-1}-R$ cannot contain any interior points of $K\cup \dom I(\Ind_P+t)$ and therefore also no interior points of $\dom I(\Ind_P+t)$. Thus, one can see infinity from $p_1,\ldots,p_{2n-1}$ via the ray $-R$.
\end{proof}
 
By the last result together with Lemma~\ref{le:buchma_valentine} we obtain the following proposition.
 
\begin{proposition}
\label{prop:isometries_parallelotopes}
Let $\zeta\in \MomO$. If $I:(\CVcn,\delta_\zeta)\to(\CVcn,\delta_\zeta)$ is an isometry, then $\dom I(\Ind_P+t)$ is a parallelotope for every parallelotope $P\in\Kn$ and $t\in\R$. 
\end{proposition}

The next Lemma is a consequence of Proposition~\ref{prop:isometries_parallelotopes} and Lemma~\ref{le:isometry_preserves_order_cup_cap}.

\begin{lemma}
\label{le:parallelotopes_neighbors}
Let $\zeta\in\MomO$ and let $I:(\CVcn,\delta_\zeta)\to(\CVcn,\delta_\zeta)$ be an isometry. If two parallelotopes $P,Q\in\Kn$ are neighbors, then also the parallelotopes $\dom I(\Ind_P+t)$ and $\dom I(\Ind_Q+t)$ are neighbors for every $t\in\R$.
\end{lemma}

\subsection{Affinities}
The purpose of this section is to show that the map $K\mapsto \dom I(\Ind_K+t)$, $K\in\Knn$, is an affinity for every isometry $I$ on $(\CVcn,\delta_\zeta)$ and $t\in\R$. For this part we follow Gruber \cite{gruber_mathematika_1978} even more closely than in the previous subsections. Here, only very minor modifications of Gruber's proof are needed which is in contrast to the previous sections where also new arguments were added.

\medskip

In the following we say that three parallelotopes $P_1,P_2,P_3\in\Kn$ form a \emph{pile} if $P_2=P_1+b$ and $P_3=P_1+2b$, where $b\in\Rn$ is a vector that is parallel to a one-dimensional edge of $P_1$ and of the same length. It is easy to see that three parallelotopes $P_1,P_2,P_3\in\Kn$ form a pile if and only if $P_1,P_2$ and $P_2,P_3$ are neighbors, $V_n(P_1\cap P_3)=\emptyset$ and $P_1\cup P_2\cup P_3$ is convex.

\begin{lemma}
\label{le:parallelotopes_piles}
Let $\zeta\in \MomO$ and let $I:(\CVcn,\delta_\zeta)\to(\CVcn,\delta_\zeta)$ be an isometry. If three parallelotopes $P_1,P_2,P_3\in\Kn$ form a pile, then also the parallelotopes\linebreak$\dom I(\Ind_{P_1}+t), \dom I(\Ind_{P_2}+t),\dom I(\Ind_{P_3}+t)$ form a pile for every $t\in\R$.
\end{lemma}
\begin{proof}
Throughout the proof fix an arbitrary pile $P_1,P_2,P_3\in\Kn$ and $t\in\R$. By Lemma~\ref{le:parallelotopes_neighbors} the parallelotopes $\dom I(\Ind_{P_1}+t), \dom I(\Ind_{P_2}+t)$ as well as $\dom I(\Ind_{P_2}+t), \dom I(\Ind_{P_3}+t)$ are neighbors. Since $V_n(P_1\cap P_3)=0$ it follows from Lemma~\ref{le:isometry_preserves_order_cup_cap} \ref{it:measure_cap} that also
$$V_n(\dom I(\Ind_{P_1}+t) \cap \dom I(\Ind_{P_3}+t))=0.$$
Furthermore, by \ref{it:pres_wedge} of same lemma, $I(\Ind_{P_1}+t)\wedge I(\Ind_{P_2}+t) \wedge I(\Ind_{P_3}+t)$ is convex and thus the same is also true for the union of the respective domains. Thus, $\dom I(\Ind_{P_1}+t), \dom I(\Ind_{P_2}+t),\dom I(\Ind_{P_3}+t)$ form a pile.
\end{proof}

\begin{lemma}
\label{le:p_q}
Let $\zeta\in \MomO$, let $I:(\CVcn,\delta_\zeta)\to(\CVcn,\delta_\zeta)$ be an isometry and let $t\in\R$. If the parallelotope $P\in\Kn$ can be written as
$$P=\{\alpha_1 b_1+\cdots + \alpha_n b_n\colon \alpha_i \in [0,1]\}$$
with $b_1,\ldots,b_n\in\Rn$, then the parallelotope $Q=\dom I(\Ind_P+t)$ can be represented as
$$Q=\{q+ \alpha_1 c_1+\cdots + \alpha_n c_n \colon \alpha_i \in [0,1]\}$$
with $q,c_1,\ldots,c_n\in\Rn$ such that
\begin{equation}
\label{eq:q_r_i_c_i}
\dom I(\Ind_{P+z_1b_1+\cdots z_n b_n}+t)=Q+z_1 c_1+\cdots +z_n c_n
\end{equation}
for every $z_1,\ldots,z_n\in\Z$.
\end{lemma}
\begin{proof}
For each $i\in \{1,\ldots,n\}$ the three parallelotopes $P-b_i, P, P+b_i$ form a pile. Thus, by Lemma~\ref{le:parallelotopes_piles} their images under the map $P\mapsto \dom I(\Ind_P+t)$ form a pile too and can be written in the form $Q-c_i, Q, Q+c_i$. Here, $c_i$ is a vector that is parallel to a one-dimensional edge of $Q$ and is of the same length. By Lemma~\ref{le:isometry_preserves_order_cup_cap} \ref{it:measure_cap} the $n$ different pairs of neighbors $P\pm b_i$, $i\in\{1,\ldots,n\}$, of $P$ are mapped to $n$ different pairs of neighbors $Q\pm c_i$ of $Q$. Since $Q$ is a parallelotope, there are exactly $n$ different such pairs of neighbors which consist of all $2n$ neighbors of $Q$. Thus, all neighbors of $Q$ are given by the $2n$ parallelotopes $Q\pm c_i$, $i\in\{1,\ldots,n\}$. As a consequence, the vectors $c_1,\ldots,c_n$ are linearly independent. Since also each of them corresponds to a one-dimensional edge of $Q$, this shows that there exists $q\in\Rn$ such that
$$Q=\{q+ \alpha_1 c_1+\cdots + \alpha_n c_n \colon \alpha_i \in [0,1]\}.$$
Furthermore, by the previous arguments,
\begin{equation}
\label{eq:q_pm_c_i}
\dom I(\Ind_P+t)=Q\qquad \text{and} \qquad \dom I(\Ind_{P\pm b_i} +t)= Q \pm c_i
\end{equation}
for every $i\in\{1,\ldots,n\}$. It remains to prove \eqref{eq:q_r_i_c_i}, which we will do by induction.

\medskip

By \eqref{eq:q_pm_c_i} the statement is true whenever $z_1,\ldots,z_n\in \Z$ are such that $|z_1|+\cdots+|z_n|\in\{0,1\}$. Assume now that for given $m\in\N$, $m\geq 2$, the statement holds for all $z_1,\ldots,z_n\in\Z$ such that $|z_1|+\cdots+|z_n|\leq m-1$. We need to show that \eqref{eq:q_r_i_c_i} also holds for all $z_1,\ldots,z_n\in\Z$ with $|z_1|+\cdots+|z_n|=m$.

Without loss of generality, we will assume that $z_1,\ldots,z_k>0=z_{k+1}=\cdots=z_n$ for some $k\in\N$ and remark that the general case can be easily reduced to this case. We now distinguish three cases:
\begin{itemize}
    \item $k=1$:
        In this case $z_1=m\geq 2$ and we need to show that
        \begin{equation}
        \label{eq:case_k=1}
        \dom I(\Ind_{P+z_1 b_1}+t)=Q+z_1 c_1.
        \end{equation}
        By the induction hypothesis
        $$\dom I(\Ind_{P+(z_1-2)b_1}+t) = Q+(z_1-2)c_1 \quad \text{and} \quad \dom I(\Ind_{P+(z_1-1)b_1}+t) = Q+(z_1-1)c_1.$$
        Since the parallelotopes $P+(z_1-2)b_1,P+(z_1-1)b_1,P+z_1b_1$ form a pile, it follows from Lemma~\ref{le:parallelotopes_piles} that also their images under the map $P\mapsto \dom(\Ind_P+t)$ form a pile. This can only be the case if \eqref{eq:case_k=1} holds.
    \item $k=2$:
        By our assumptions, $z_1\geq 1$ and $z_2\geq 1$ and the parallelotopes $P+(z_1-1)b_1+z_2b_2$ and $P+z_1 b_1 + (z_2-1)b_2$ are neighbors of $P+z_1 b_1+z_2 b_2$. By Lemma~\ref{le:parallelotopes_neighbors} also their respective images under the map $P\mapsto \dom I(\Ind_P+t)$ are neighbors. Thus, by the induction hypothesis, the parallelotopes
        $$\dom I(\Ind_{P+(z_1-1)b_1+z_2b_2}+t)=Q+(z_1-1)c_1+z_2 c_2$$
        and
        $$\dom I(\Ind_{P+z_1b_1+(z_2-1)b_2}+t) = Q+z_1 c_1+(z_2-1)c_2$$
        are neighbors of the parallelotope $\dom I(\Ind_{P+z_1 b_1+z_2 b_2}+t)$. The former two parallelotopes have only two common neighbors, namely,
        $$Q+(z_1-1)c_1+(z_2-1)c_2\quad \text{and}\quad Q+z_1 c_1+z_2 c_2.$$
        Since by the induction hypothesis
        $$\dom I(\Ind_{P+(z_1-1)b_1+(z_2-1)b_2}+t)= Q+(z_1-1)c_1+(z_2-1)c_2$$
        and since by Lemma~\ref{le:isometry_preserves_order_cup_cap} \ref{it:measure_cap}
        $$V_n(\dom I(\Ind_{P+(z_1-1)b_1+(z_2-1)b_2}+t) \cap \dom I(\Ind_{P+z_1 b_1+z_2 b_2}+t))=0,$$
        it must follow that
        $$\dom I(\Ind_{P+z_1 b_1 + z_2 b_2}+t)=Q+z_1 c_1+z_2 c_2.$$
    \item $k\geq 3$:
        We argue similar as in the case $k=2$. The parallelotopes
        $$P+(z_1-1)b_1+z_2 b_2+\cdots + z_k b_k, \cdots, P+z_1 b_1+ \cdots + z_{k-1}b_{k-1}+(z_k-1)b_k$$
        are neighbors of the parallelotope $P+z_1 b_1+\cdots+z_k b_k$. If follows from Lemma~\ref{le:parallelotopes_neighbors} and the induction hypothesis that the parallelotopes
        \begin{align*}
        \dom I(\Ind_{P+(z_1-1)b_1+z_2 b_2+\cdots + z_k b_k}+t)&=Q+(z_1-1)c_1+z_2 c_2+\cdots + z_k b_k\\
        &\vdots\\
        \dom I(\Ind_{P+z_1 b_1+ \cdots + z_{k-1}b_{k-1}+(z_k-1)b_k}+t) &= Q+z_1 c_1+\cdots + z_{k-1} b_{k-1}+(z_k-1)c_k
        \end{align*}
        are neighbors of the parallelotope $\dom I(\Ind_{P+z_1 b_1+\cdots +z_k b_k}+t)$. Since the only common neighbor of the former is $Q+z_1 c_1+\cdots + z_k c_k$, it follows that
        $$\dom I(\Ind_{P+z_1 b_1+\cdots z_k b_k}+t)= Q+z_1 c_1+\cdots + z_k c_k,$$
        which completes the proof.\qedhere
\end{itemize}
\end{proof}

The following result generalizes the last lemma.

\begin{lemma}
\label{le:p_q_refined}
Let $\zeta\in \MomO$, let $I:(\CVcn,\delta_\zeta)\to(\CVcn,\delta_\zeta)$ be an isometry and let $t\in\R$. Furthermore let the parallelotopes $P,Q\in\Kn$ and $b_1,\ldots,b_n,q,c_1,\ldots,c_n\in\Rn$ be as in Lemma~\ref{le:p_q}. If for $m\in\N$
$$P^{(m)}:=\frac{1}{2^m}P=\left\{\alpha_1 b_1 + \cdots + \alpha_n b_n \colon \alpha_i \in \left[0,\frac{1}{2^m}\right]\right\},$$
and $Q^{(m)}:=\dom I(\Ind_{P^{(m)}}+t)$, then
$$Q^{(m)} = \left\{q+\alpha_1 c_1 + \cdots + \alpha_n c_n\colon \alpha_i \in \left[0,\frac{1}{2^m}\right]\right\}.$$
Moreover,
$$\dom I(\Ind_{P^{(m)}+\frac{1}{2^m}(z_1 b_1+\cdots+z_n b_n)}+t)=Q^{(m)}+\frac{1}{2^m} \left(z_1 c_1 + \cdots + z_n c_n \right)$$
for every $m\in\N$ and $z_1,\ldots,z_n\in\Z$.
\end{lemma}
\begin{proof}
It is enough to prove the result for the case $m=1$. The general case follows by applying the Lemma iteratively.

By Lemma~\ref{le:p_q}, applied to $P^{(1)}$ and $Q^{(1)}$, there exist $q', c_1',\ldots,c_n'\in\Rn$ such that
$$Q^{(1)}=\{q'+\alpha_1 c_1' + \cdots + \alpha_n c_n'\colon \alpha_i\in[0,1]\}$$
and
\begin{equation}
\label{eq:dom_i_p_1_eq_q_1}
\dom I(\Ind_{P^{(1)}+\frac 12 (z_1 b_1+\cdots + z_n b_n)}+t) = Q^{(1)} + z_1 c_1' + \cdots + z_n c_n'
\end{equation}
for every $z_1,\ldots,z_n\in\Z$.

Observe that $P^{(1)} + \frac 12 (2 z b_i)\subset P + z b_i$ for every $i\in\{1,\ldots,n\}$ and $z\in\Z$. Thus, by \eqref{eq:dom_i_p_1_eq_q_1}, Lemma~\ref{le:isometry_preserves_order_cup_cap} \ref{it:strict_incl} and Lemma~\ref{le:p_q},
\begin{equation}
\label{eq:q_1_subset_q}
Q^{(1)}+2 z c_i' = \dom I(\Ind_{P^{(1)} + \frac 12 (2 z b_i)}+t) \subseteq \dom I(\Ind_{P + z b_i}+t) = Q + z c_i
\end{equation}
for every $i\in\{1,\ldots,n\}$ and $z\in\Z$, which can only be true if $2 c_i' = c_i$. Hence,
$$Q^{(1)} =\left\{q'+\alpha_1 c_1 + \cdots + \alpha_n c_n \colon \alpha_i \in \left[0,\frac 12\right] \right\}$$
and
$$\dom I(\Ind_{P^{(1)}+\frac 12 (z_1 b_1+\cdots + z_n b_n)}+t) = Q^{(1)} + \frac 12(z_1 c_1 + \cdots + z_n c_n)$$
for every $z_1,\ldots,z_n\in\Z$.

It remains to show $q=q'$. Observe, that $P^{(1)}-\frac 12 b_i \subset P-b_i$, $i\in\{1,\ldots,n\}$, are neighbors of $P^{(1)} \subset P$. Thus, by Lemma~\ref{le:parallelotopes_neighbors} and \eqref{eq:q_1_subset_q}, also $Q^{(1)} - \frac 12 c_i \subset Q - c_i$ are neighbors of $Q^{(1)} \subset Q$, which can only be the case if $q=q'$.
\end{proof}

\begin{proposition}
\label{prop:affinity}
Let $\zeta\in \MomO$ and let $I:(\CVcn,\delta_\zeta)\to(\CVcn,\delta_\zeta)$ be an isometry. There exist $\phi\in\GLn$ and $x_0\in\Rn$ such that
$$\dom I(\Ind_K+t) = \phi K +x_0$$
for every $K\in\Knn$ and $t\in\R$.
\end{proposition}
\begin{proof}
Throughout the proof we fix an arbitrary parallelotope
$$P=\{\alpha_1 b_1 + \cdots + \alpha_n b_n \colon \alpha_i \in [0,1]\}$$
with $b_1,\ldots,b_n\in\Rn$. By Lemma~\ref{le:p_q} there exist $q_t,c_{1,t},\ldots,c_{n,t}\in\Rn$ (using the same numeration as in Lemma~\ref{le:p_q}), depending on $t\in\R$, such that
$$\dom I(\Ind_P+t)=\{q_t+\alpha_1 c_{1,t}+\cdots+\alpha_n c_{n,t}\colon: \alpha_i \in [0,1]\}.$$
Define $\phi_t\in\GLn$ as the unique general linear transformation such that $\phi_t b_i= c_{i,t}$ for $i\in\{1,\ldots,n\}$. By Lemma~\ref{le:p_q_refined} we now have
\begin{align}
\begin{split}
\label{eq:dom_i_ind_p_equals_phi_t_p}
\dom I\left(\Ind_{\frac{1}{2^m} P + \frac{1}{2^m}(z_1 b_1 + \ldots + z_n b_n)} +t\right) &= q_t + \phi_t\left( \frac{1}{2^m} P + \frac{1}{2^m}(z_1 b_1 + \ldots + z_n b_n)\right)\\
&= q_t + \frac{1}{2^m} \left(\phi_t P + z_1 c_{1,t} + \ldots + z_n c_{n,t}\right)
\end{split}
\end{align}
for every $m\in\N$, $z_1,\ldots,z_n\in\Z$ and $t\in\R$.

Fix arbitrary $t_1,t_2\in\R$ such that $t_1<t_2$. Since
$$\epi\left(\Ind_{\frac{1}{2^m} P + \frac{1}{2^m}(z_1 b_1 + \ldots + z_n b_n)}+t_2\right) \subset \epi \left(\Ind_{\frac{1}{2^m} P + \frac{1}{2^m}(z_1 b_1 + \ldots + z_n b_n)} + t_1\right)$$
it follows from Lemma~\ref{le:isometry_preserves_order_cup_cap} \ref{it:strict_incl} together with \eqref{eq:dom_i_ind_p_equals_phi_t_p} that
$$q_{t_2}+\frac{1}{2^m} \big(\phi_{t_2} P + z_1 c_{1,t_2} + \ldots + z_n c_{n,t_2}\big) \subseteq q_{t_1}+\frac{1}{2^m} \big(\phi_{t_1} P + z_1 c_{1,t_1} + \ldots + z_n c_{n,t_1}\big)$$
for every $m\in\N$ and $z_1,\ldots,z_n\in\Z$. Choosing $m$ large enough shows that this can only be the case if $q_{t_1}=q_{t_2}$. Similarly, choosing $z_i$ large enough shows that $c_{i,t_1}=c_{i,t_2}$ for $i\in\{1,\ldots,n\}$. Thus, there exist $q\in\Rn$ and $\phi\in\GLn$ such that
\begin{equation}
\label{eq:dom_i_ind_p_equals_phi_p}
\dom I\left(\Ind_{\frac{1}{2^m} P + \frac{1}{2^m}(z_1 b_1 + \ldots + z_n b_n)} +t\right) = q + \phi\left( \frac{1}{2^m} P + \frac{1}{2^m}(z_1 b_1 + \ldots + z_n b_n)\right)
\end{equation}
for every $m\in\N$ and $z_1,\ldots,z_n\in\Z$ and $t\in\R$. In particular, $q$ and $\phi$ do not depend on $t$.

\medskip

Now let a convex body $K\in\Knn$ be given. For every $m\in\N$ and $z_1,\ldots,z_n\in\Z$ such that
$$\frac{1}{2^m} P + \frac{1}{2^m}(z_1 b_1+ \cdots + z_n b_n) \subset K$$
it follows from Lemma~\ref{le:isometry_preserves_order_cup_cap} \ref{it:strict_incl} and \eqref{eq:dom_i_ind_p_equals_phi_p} that
$$q + \phi\left( \frac{1}{2^m} P + \frac{1}{2^m}(z_1 b_1 + \ldots + z_n b_n)\right) \subseteq \dom I(\Ind_K +t)$$
for every $t\in\R$. Similarly, for every $m\in\N$ and $z_1,\ldots,z_n\in\Z$ such that
$$V_n\left(\left(\frac{1}{2^m} P + \frac{1}{2^m}(z_1 b_1+ \cdots + z_n b_n) \right) \cap K\right) = 0$$
we have
$$\Psi_{\zeta}\left(\epi\left(\Ind_{\frac{1}{2^m} P + \frac{1}{2^m}(z_1 b_1+ \cdots + z_n b_n)}+t\right)\cap \epi(\Ind_K +t)\right)=0$$
and thus, by Lemma~\ref{le:isometry_preserves_order_cup_cap} \ref{it:measure_cap},
$$\Psi_{\zeta}\left(\epi I\left(\Ind_{\frac{1}{2^m} P + \frac{1}{2^m}(z_1 b_1+ \cdots + z_n b_n)}+t\right)\cap \epi I(\Ind_K +t)\right)=0$$
and therefore, by \eqref{eq:dom_i_ind_p_equals_phi_p},
$$V_n\left(\left(q + \phi\left( \frac{1}{2^m} P + \frac{1}{2^m}(z_1 b_1 + \ldots + z_n b_n)\right) \right) \cap \dom I(\Ind_K +t)  \right)=0$$
for every $t\in\R$. The last two statements together imply that $\dom I(\Ind_K+t)=q +\phi K$ for every $t\in\R$, which completes the proof.
\end{proof}

\subsection{Partitions}
For the final steps of our proof we will need partitions of convex bodies. In particular, we will use notions and ideas that were presented in \cite{cavallina_colesanti} in order to classify valuations on $\CVc$.

\medskip

Let $K,K_1,\ldots,K_m\in\Knn$ be convex bodies. The set $\Gamma_K:=\{K_1,\ldots,K_m\}$ is called a \emph{convex partition} of $K$ if
$$K=\bigcup_{i=1}^m K_i\qquad\text{and}\qquad \interior(K_i \cap K_j) = \emptyset$$
for every $i,j\in\{1,\ldots,m\}$ with $i\neq j$. A convex partition is furthermore called an \emph{inductive partition} if there exist $H_1,\ldots,H_l\in\Knn$ such that $H_l=K$ and such that for every $i\in\{1,\ldots,l\}$ either $H_i\in\Gamma_K$ or there exist $j,k<i$ such that $H_i=H_j\cup H_k$ and $\interior (H_j\cap H_k)=\emptyset$.

\medskip

It is obvious from the definition that if $\Gamma_K$ is an inductive partition of $K$, then there exist $K',K''\in\Knn$ such that $K=K'\cup K''$ and $\interior(K'\cap K'')=\emptyset$. Furthermore, $\Gamma_K$ contains inductive partitions for the sets $K'$ and $K''$ and therefore, they can be split themselves in a similar away. This process can then be repeated until finally, after a finite number of steps, the sets $K_1,\ldots,K_m$ are obtained.

Vice versa, an inductive partition $\Gamma_K$ of $K$ allows us to create a sequence of unions of sets, starting with the sets $K_1,\ldots,K_m$, where each union is again a convex body such that after a finite number of steps the set $K$ is obtained.

\medskip

The proof of the following result is obtained from the proof of \cite[Lemma 7.5]{cavallina_colesanti}, using only slight modifications.

\begin{lemma}
\label{le:partition_decomposition}
Let $\zeta\in \MomO$, let $I:(\CVcn,\delta_\zeta)\to(\CVcn,\delta_\zeta)$ be an isometry and let $K\in\Knn$. If $\Gamma_K=\{K_1,\ldots,K_m\}$ is an inductive partition of $K$, then
$$I(u+\Ind_K)=I(u+\Ind_{K_1})\wedge \cdots\wedge I(u+\Ind_{K_m})$$
for every $u\in\CVcn$ and
$$\interior \left(\dom I(u+\Ind_{K_i}) \cap \dom I(u+\Ind_{K_j})\right) =\emptyset$$
for every $i\neq j$.
\end{lemma}
\begin{proof}
Let $\Gamma_K$ be given an let $H_1,\ldots,H_l$ be as in the definition of an inductive partition. In particular, $H_l=K$. We will prove the statement by induction on $l\in\N$. The case $l=1$ is trivial and there is nothing to show. Therefore, let $l\geq 2$ and assume that the statement is true for $l-1$. Observe that $H_l\not\in\Gamma_K$. Otherwise $K=H_l\in \Gamma_K$ which can only be the case if $\Gamma_K=\{K\}$ and therefore, $l=1$. Hence, there exist $j,k<l$ such that $H_j\cup H_k = K$ and $\interior (H_j\cap H_k)=\emptyset$. Since both $H_j$ and $H_k$ are convex bodies with non-empty interiors it follows that $u+\Ind_{H_j},u+\Ind_{H_k}\in\CVcn$ and furthermore
\begin{equation}
\label{eq:u_ind_h_j_wedge_u_ind_h_k}
(u+\Ind_{H_j})\wedge (u+\Ind_{H_k}) = u+ \Ind_{H_j \cup H_k} = u+\Ind_K
\end{equation}
and
$$(u+\Ind_{H_j})\vee (u+\Ind_{H_k}) = u+\Ind_{H_j\cap H_k}.$$
In particular, the last function has a domain with empty interior and therefore
\begin{equation}
\label{eq:psi_zeta_epi_u_ind_h_j_cap_u_ind_h_k}
\Psi_{\zeta}(\epi(u+\Ind_{H_j}) \cap \epi(u+\Ind_{H_k}))=0.
\end{equation}
Hence, it follows from \eqref{eq:u_ind_h_j_wedge_u_ind_h_k} and Lemma~\ref{le:isometry_preserves_order_cup_cap} \ref{it:pres_wedge} that
$$I(u+\Ind_K)=I(u+\Ind_{H_j})\wedge I(u+\Ind_{H_k})$$
and similarly by \eqref{eq:psi_zeta_epi_u_ind_h_j_cap_u_ind_h_k} and Lemma~\ref{le:isometry_preserves_order_cup_cap} \ref{it:measure_cap} that
$$\Psi_{\zeta}(\epi I(u+\Ind_{H_j}) \cap \epi I(u+\Ind_{H_k}))=0$$
which implies
$$\interior \left(\dom I(u+\Ind_{H_j}) \cap \dom I(u+\Ind_{H_k}) \right)=\emptyset.$$
We now apply the induction assumption to the inductive partitions $\Gamma_{H_j}=\{L\in \Gamma_K\colon L\subseteq H_j\}$ and $\Gamma_{H_k}=\{L\in \Gamma_K\colon L\subseteq H_k\}$ and obtain
$$I(u+\Ind_K)=I(u+\Ind_{H_j})\wedge I(u+\Ind_{H_k}) = \bigwedge_{L\in \Gamma_{H_j}}I(u+\Ind_{L}) \wedge \bigwedge_{L\in\Gamma_{H_k}} I(u+\Ind_{L})$$
such that the intersection of the domains of each two the functions on the right-hand side has empty interior.
\end{proof}

The next result is easy to see (see also \cite[Section 7.2]{cavallina_colesanti}).

\begin{lemma}
\label{le:polytope_partition}
For every $K\in\Knn$ and for every $\varepsilon>0$ there exists an inductive partition $\Gamma_K=\{K_1,\ldots,K_m\}$ of $K$ such that $\diam K_i\leq \varepsilon$ for every $i\in\{1,\ldots,m\}$.
\end{lemma}

Note that if $\Gamma_K=\{K_1,\ldots,K_m\}$ is an inductive partition of $K\in\Knn$, then it is easy to see that $\{\alpha(K_1),\ldots,\alpha(K_m)\}$ is an inductive partition of $\alpha(K)$ for every affine transform $\alpha:\Rn\to\Rn$ with non-vanishing determinant.

\begin{lemma}
\label{le:iso_f}
Let $\zeta\in \MomO$ and let $I:(\CVcn,\delta_\zeta)\to(\CVcn,\delta_\zeta)$ be an isometry. If $\phi\in\GLn$ and $x_o\in\R$ are as in Proposition~\ref{prop:affinity}, then
$$I(\Ind_K+t)=\Ind_{\phi K + x_o}+f(t)$$
for every $K\in\Knn$ and every $t\in\R$, where $f:\R\to\R$ is given by
\begin{equation}
\label{eq:def_f}
f(t)=\zeta^{-1}\left(\frac{\zeta(t)}{|\det \phi|}\right)
\end{equation}
for $t\in\R$. In particular, $\phi$ is such that the right-hand side of \eqref{eq:def_f} is well-defined.
\end{lemma}
\begin{proof}
Fix arbitrary $K\in\Knn$ and $t\in\R$ and let $u=I(\Ind_K+t)\in\CVcn$. It follows from Proposition~\ref{prop:affinity} that $\dom u =\phi K +x_0=:L$ and thus
$$u=u+\Ind_L.$$
We will first show that the restriction of $u$ to $L$ is constant. Assume on the contrary, that $u$ is not constant on $L$. In this case there exist $y_0,y_1\in\interior L$ and $s_0,s_1\in\R$ such that
$$u(y_0)=s_0 < s_1=u(y_1).$$
Since a convex functions is continuous on the interior of its domain, there exists $\varepsilon>0$ such that $B(y_0,\varepsilon)\subset L$ and $B(y_1,\varepsilon)\subset L$ and furthermore
\begin{equation}
\label{eq:u_y0_y1_s0_s1}
u(y)<\frac{s_0+s_1}{2}\quad \forall y\in B(y_0,\varepsilon) \qquad \text{and} \qquad  u(y)>\frac{s_0+s_1}{2}\quad \forall y\in B(y_1,\varepsilon).
\end{equation}
By Lemma~\ref{le:polytope_partition} there exists an inductive partition $\Gamma_{L}=\{L_1,\ldots,L_m\}$ of $L$ such that $\diam L_i<\varepsilon$ for every $i\in\{1,\ldots,m\}$. Furthermore, defining $\alpha(x)=\phi^{-1}(x-x_0)$ for $x\in\Rn$, we have $\alpha(L) = K$ and an inductive partition $\Gamma_K=\{K_1,\ldots,K_m\}$ of $K$ is obtained by setting $K_i=\alpha(L_i)$ for $i\in\{1,\ldots,m\}$. Now if $u_i:=I(\Ind_{K_i}+t)\in\CVcn$, then it follows from Proposition~\ref{prop:affinity} that $u_i=u_i+\Ind_{L_i}$ for every $i\in\{1,\ldots,m\}$. Thus, by Lemma~\ref{le:partition_decomposition}, 
$$u+\Ind_L = I(\Ind_K+t)=I(\Ind_{K_1}+t)\wedge \cdots \wedge I(\Ind_{K_m}+t)=(u_1+\Ind_{L_1})\wedge \cdots \wedge (u_m+\Ind_{L_m})$$
where the intersection of the domains of each two functions on the right-hand side has dimension less or equal to $n-1$. In particular, this implies that $u_i(y)=u(y)$ for every $y\in L_i$.

Next, let $i_0,i_1\in\{1,\ldots,m\}$ be such that $y_0\in L_{i_0}$ and $y_1\in L_{i_1}$. By our assumptions on $\Gamma_L$ we have $L_{i_0} \subset B(y_0,\varepsilon)$ and therefore, by \eqref{eq:u_y0_y1_s0_s1}, $$u_{i_0}(y)=u(y)< \frac{s_0+s_1}{2}$$
for every $y\in L_{i_0}$. Hence,
$$\Psi_{\zeta}(\epi u_{i_0})> \zeta\left(\frac{s_0+s_1}{2}\right) V_n(L_{i_0}) = |\det \phi|\,\zeta\left(\frac{s_0+s_1}{2}\right) V_n(K_{i_0}).
$$
Since by Proposition~\ref{prop:iso_measure_pres} also
$$\Psi_{\zeta}(\epi u_{i_0}) = \Psi_{\zeta}(\epi (\Ind_{K_{i_0}}+t))=\zeta(t) V_n(K_{i_0})$$
we therefore obtain
$$\zeta(t)>|\det \phi |\,\zeta\left(\frac{s_0+s_1}{2}\right).$$
Similarly, repeating the last steps with $i_1$, we obtain
$$\zeta(t)<|\det \phi|\,\zeta\left(\frac{s_0+s_1}{2}\right)$$
which is a contradiction. Thus, we conclude that $u$ is constant on its domain.

\medskip

By the first part of the proof there exists $s\in\R$ such that
$$I(\Ind_K+t)=\Ind_{\phi K +x_0}+s.$$
Using Proposition~\ref{prop:iso_measure_pres} again now shows
\begin{align*}
\zeta(t)V_n(K)&= \Psi_{\zeta}(\epi(\Ind_K+t))\\
&=\Psi_{\zeta}(\epi(\Ind_{\phi K+x_0}+s))\\
&= |\det\phi|\,\zeta(s) V_n(K)
\end{align*}
and thus
$$s=\zeta^{-1}\left(\frac{\zeta(t)}{|\det \phi|}\right).$$
In particular, $s$ does not depend on $K$, which completes the proof.
\end{proof}

\begin{remark}
Note that the last result implies that if $\lim_{t\to\infty}\zeta(-t)=A<\infty$, then $|\det \phi|\geq 1$. Otherwise, we could find $t\in\R$ such that $\zeta(t)/|\det \phi|>A$, in which case \eqref{eq:def_f} is not well-defined.
\end{remark}

\subsection{Proof of Theorem~\ref{thm:class_iso}}
We will use ideas from the proof of \cite[Theorem 8.1]{cavallina_colesanti}.

\bigskip

Let $\phi\in\Phi(\zeta)$ and $x_0\in\Rn$ be given and let $\alpha(x)=\phi(x)+x_0$ for $x\in\Rn$ and $f(t)=\zeta^{-1}(\zeta(t)/|\det \phi|)$ for $t\in\R$. It is easy to see that $u\circ \alpha^{-1}\in\CVcn$ for every $u\in\CVcn$. Furthermore, since $\zeta$ is strictly decreasing with $\lim_{t\to\infty}\zeta(t)=0$ the function $f$ is strictly increasing with $\lim_{t\to\infty} f(t)=+\infty$. Moreover, since $f$ is convex it follows that
\begin{align*}
f(u(\lambda x+(1-\lambda)y))&\leq f(\lambda u(x) + (1-\lambda) u(y))\\
&\leq \lambda f(u(x)) + (1-\lambda) f(u(y))
\end{align*}
for every $x,y\in \Rn$, $0\leq \lambda \leq 1$ and every convex function $u:\Rn\to(-\infty,+\infty]$. Hence, $I(u):=f(u\circ \alpha^{-1})\in\CVcn$ for every $u\in\CVcn$. Now for every $u,v\in\CVcn$
\begin{align*}
\delta_{\zeta}(I(u),I(v)) &= \int_{\Rn} |\zeta(I(u)(x))-\zeta(I(v)(x))|\d x\\
&= \int_{\Rn} |\zeta(f(u(\phi^{-1} (x-x_0))))-\zeta(f(v(\phi^{-1}(x-x_0))))|\d x\\
&= |\det \phi| \int_{\Rn} |\zeta(f(u(x)))-\zeta(f(v(x)))|\d x\\
&= |\det \phi| \int_{\Rn} \left|\frac{\zeta(u(x))}{|\det \phi|}-\frac{\zeta(v(x))}{|\det \phi|} \right| \d x\\
&= \int_{\Rn} |\zeta(u(x))-\zeta(v(x))| \d x\\
&= \delta_{\zeta}(u,v)
\end{align*}
which shows that $I$ is an isometry.

\medskip

Conversely, let an isometry $I:(\CVcn,\delta_\zeta)\to(\CVcn,\delta_\zeta)$ be given and fix an arbitrary $u\in\CVcn$. We will first consider the case that $\dom u = K\in\Knn$ and in particular that the restriction of $u$ to $K$ is bounded. Let $m_0=\min_{x\in K} u(x)$ and $m_1=\max_{x\in K} u(x)$. By Lemma~\ref{le:iso_f} there exist $\phi\in\GLn$ and $x_0\in\R$ such that
$$I(\Ind_K+m_0) = \Ind_{L} + f(m) \qquad \text{and} \qquad I(\Ind_K+m_1)=\Ind_{L}+f(m_1)$$
where $L=\phi K+x_0$ and $f(t)=\zeta^{-1}(\zeta(t)/|\det \phi|)$ for $t\in\R$. In particular, $f$ is well-defined and, by the properties of $\zeta$, the function $f$ is strictly increasing with $\lim_{t\to\infty} f(t)=+\infty$. Since
$$\Ind_K+m_0 \leq u \leq \Ind_K+m_1$$
it follows from Lemma~\ref{le:isometry_preserves_order_cup_cap} \ref{it:strict_incl} that
$$\Ind_{L}+f(m_0)\leq I(u)\leq \Ind_{L}+f(m_1).$$
In particular, $\dom I(u) =L$ and the restriction of $I(u)$ to $L$ is continuous and bounded.

Fix an arbitrary $\varepsilon>0$. We will show that 
\begin{equation}
\label{eq:diff_i_u_f_u_a_inv}
|I(u)-f(u\circ \alpha^{-1})|<\varepsilon
\end{equation}
pointwise, where $\alpha(x)=\phi x + x_0$ for $x\in\Rn$. Since $\varepsilon$ is arbitrary, this then implies that $I(u)=f(u\circ \alpha^{-1})$. The fact that $f$ is convex then follows easily from the fact that $f(u\circ \alpha^{-1})=I(u)\in\CVcn$ for every $u\in\CVcn$.

First observe that $f(u(\alpha^{-1} (x)))=+\infty$ if and only if $u(\alpha^{-1}(x))=+\infty$ if and only if $\phi^{-1}(x-x_0) \not\in K$ if and only if $x\not\in L$. Hence $f(u\circ \alpha^{-1})\equiv I(u)$ pointwise on $\Rn\backslash L$.

Since $f$ is continuous, its restriction to the compact interval $[m_0,m_1]$ is uniformly continuous. Hence, there exists $\gamma>0$ such that
$$|f(s)-f(t)|< \varepsilon$$
for every $s,t\in[m_0,m_1]$ with $|s-t|<\gamma$. Similarly, the restriction of $u$ to $K$ is uniformly continuous. Hence, there exists $\eta>0$ such that
$$|u(x)-u(y)|<\gamma$$
for every $x,y\in K$ with $|x-y|<\eta$ and thus, since $u(x),u(y)\in[m_0,m_1]$,
$$|f(u(x))-f(u(y))|<\varepsilon.$$
By Lemma~\ref{le:polytope_partition} there exists an inductive partition $\Gamma_K=\{K_1,\ldots,K_m\}$ of $K$ such that $\diam P_i\leq \eta$ for every $i\in\{1,\ldots,m\}$. Furthermore, denote $m_{0,i}=\min_{x\in K_i} u(x)$ and $m_{1,i}=\max_{x\in K_i} u(x)$ for $i\in\{1,\ldots,m\}$. By the choice of our partition we now have
\begin{equation}
\label{eq:diff_f_m_i}
f(m_{1,i})-f(m_{0,i})=|f(m_{1,i})-f(m_{0,i})|<\varepsilon
\end{equation}
for every $i\in\{1,\ldots,m\}$. Moreover, a corresponding inductive partition $\Gamma_L=\{L_1,\ldots,L_m\}$ of $L$ is obtained by taking $L_i=\alpha(K_i)$ for every $i\in\{1,\ldots,m\}$. By Lemma~\ref{le:partition_decomposition} we now have
\begin{align}
\begin{split}
\label{eq:i_u_part_q_i_p_i}
(I(u)+\Ind_{L_1}) \wedge \cdots \wedge (I(u)+\Ind_{L_m}) &= I(u)+\Ind_L\\
&= I(u)\\
&= I(u+\Ind_K)\\
&= I(u+\Ind_{K_1})\wedge \cdots \wedge I(u+\Ind_{K_m})
\end{split}
\end{align}
where for the first and last expression the intersections of the domains of each two functions have empty interior. Since $\Ind_{K_i}+m_{0,i}\leq \Ind_{K_i}+u \leq \Ind_{K_i}+m_{1,i}$ it follows from Lemma~\ref{le:iso_f} and Lemma~\ref{le:isometry_preserves_order_cup_cap} \ref{it:strict_incl}, similarly as before, that
$$\Ind_{L_i}+f(m_{0,i}) \leq I(u+\Ind_{K_i})\leq \Ind_{L_i}+f(m_{1,i})$$
for every $i\in\{1,\ldots,m\}$. In particular, $\dom I(u+\Ind_{K_i}) = L_i$ and thus, by \eqref{eq:i_u_part_q_i_p_i}, $I(u+\Ind_{K_i})=I(u)+\Ind_{L_i}$. Hence,
$$\Ind_{L_i}+f(m_{0,i})\leq I(u)+\Ind_{L_i} \leq \Ind_{L_i}+f(m_{1,i})$$
for every $i\in\{1,\ldots,m\}$. Since $x\in L_i$ implies $\alpha^{-1} (x) \in K_i$ we furthermore have, by the definition of $m_{0,i}$ and $m_{1,i}$,
$$f(m_{0,1}) \leq f(u(\alpha^{-1} (x))) \leq f(m_{1,i})$$
for every $x\in L_i$. Thus, combining the last two inequalities with \eqref{eq:diff_f_m_i} we obtain
$$|I(u)(x)-f(u(\alpha^{-1}(x)))|\leq f(m_{1,i})-f(m_{0,i}) < \varepsilon$$
for every $x\in L_i$ and for every $i\in\{1,\ldots,m\}$. Since $\bigcup_{i=1}^m L_i=L$, this proves \eqref{eq:diff_i_u_f_u_a_inv}.

\medskip

Lastly, we consider the case that $u\in\CVcn$ is arbitrary. Let $K_j\in\Knn$, $j\in\N$, be an increasing sequence of convex bodies such that
$$\bigcup_{j=1}^{\infty} K_j = \dom u$$
and such that the restriction of $u$ to each of the sets $K_j$ is bounded. Consider the sequence of functions $u_j:=u+\Ind_{K_j}\in\CVcn$ for $j\in\N$. Since $u_j$ converges to $u$ pointwise on $\Rn\backslash \bd (\dom u)$, it follows from Lemma~\ref{le:epi_conv_equiv} that $u_j\eto u$ as $j\to\infty$. Since isometries are continuous with respect to epi-convergence this implies that $I(u_j)\eto I(u)$. By the first part of the proof we have
$$I(u_j)=f((u+\Ind_{K_j})\circ \alpha^{-1}) = f(u\circ \alpha^{-1})+\Ind_{\alpha(K_j)}$$
for every $j\in\N$ and thus
$$I(u_j)\eto f(u\circ \alpha^{-1})+\Ind_{\dom u} = f(u\circ \alpha^{-1})$$
as $j\to\infty$. Since the limit is unique, we conclude that $I(u)=f(u\circ \alpha^{-1})$.\hfill\qedsymbol

\section{Further Metrics}
\label{se:further_metrics}

In this section we introduce two further metric which can be seen as functional analogs of the Hausdorff metric. In Section~\ref{subse:ext_hd_metric} we define a metric on $\CVc$ by integrating over an extension of the Hausdorff metrics of the level sets of two functions. We show that convergence with respect to this metric is equivalent to epi-convergence. However, this metric is not invariant with respect to translations.

We furthermore introduce another analog of the Hausdorff metric on $\CVc$ in Section~\ref{subse:another_hd_type_metric} which is of a more geometric nature. Convergence with respect to this metric implies epi-convergence but the converse is only true for super-coercive functions. However, this metric is even invariant with respect to translations of epigraphs.

\subsection{A New Metric Based on an Extension of the Hausdorff Distance}
\label{subse:ext_hd_metric}
Since epi-convergence on $\CVc$ corresponds to Hausdorff-convergence of level sets (see\linebreak Lemma~\ref{le:hd_conv_lvl_sets}), one approach to define a new distance of two functions $u,v\in\CVc$ that is equivalent to epi-convergence, is to integrate over the Hausdorff distances of their level sets. However, this creates the problem that one also needs to define a distance between the empty set and a non-empty set.

\medskip
In the following we will call a map $\tilde{d}_H:(\Kn\cup \emptyset) \times (\Kn\cup \emptyset)\to [0,\infty]$ an \emph{extension} of the Hausdorff metric $d_H$ if $\tilde{d}_H$ is a metric and $\tilde{d}_H(K,L)=d_H(K,L)$ for every $K,L\in\Kn$. As the next result shows, there is only one way to find an extension of the Hausdorff metric that preserves translation invariance.

\begin{lemma}
\label{le:hd_ext_infty}
A metric $\tilde{d}_H$ is an extension of $d_H$ such that $\tilde{d}_H(K+x,L+x)= \tilde{d}_H(K,L)$ for every $K,L\in\Kn \cup \emptyset$ and $x\in\Rn$ if and only if
\begin{equation}
\label{eq:infty_extension_hd}
\tilde{d}_H(K,L)=\begin{cases}
d_H(K,L)\quad & K\neq \emptyset \neq L\\
0\quad & K = L = \emptyset\\
+\infty\quad& K\neq \emptyset, L=\emptyset \text{ or } K=\emptyset, L\neq \emptyset
\end{cases}
\end{equation}
for $K,L\in\Kn\cup \emptyset$.
\end{lemma}
\begin{proof}
It is easy to check that \eqref{eq:infty_extension_hd} defines a metric with the desired properties.

Conversely, let $\tilde{d}_H$ be an extension of $d_H$ such that $\tilde{d}_H(K+x,L+x)= \tilde{d}_H(K,L)$ for every $K,L\in\Kn \cup \emptyset$ and $x\in\Rn$. By the triangle inequality together with translation invariance
$$\tilde{d}_H(K,\{x\})\leq \tilde{d}_H(K,\emptyset) + \tilde{d}_H(\emptyset,\{x\}) = \tilde{d}_H(K,\emptyset) + \tilde{d}_H(\emptyset,\{0\})$$
for every $K\in\Kn$ and $x\in\Rn$. In particular, the left-hand side of this expression can be arbitrarily large and therefore the inequality implies that either $\tilde{d}_H(K,\emptyset)=+\infty$ for every $K\in\Kn$ and/or $\tilde{d}_H(\emptyset,\{0\})=+\infty$. Since
$$\tilde{d}_H(\emptyset,\{0\}) \leq \tilde{d}_H(\emptyset,K)+ \tilde{d}_H(K,\{0\})=\tilde{d}_H(\emptyset,K)+ d_H(K,\{0\})$$
for every $K\in\Kn$, also the case $\tilde{d}_H(\emptyset,\{0\})=+\infty$ implies that $\tilde{d}_H(K,\emptyset)=+\infty$. Thus, we conclude that $\tilde{d}_H$ is as in \eqref{eq:infty_extension_hd}.
\end{proof}

It is easy to see that for $u,v\in\CVc$,
$$\tilde{d}(u,v):=\int_0^{+\infty} \tilde{d}_H(\{e^{-u} \geq s\},\{e^{-v} \geq s\}) \d s$$
defines a metric on $\CVc$, where $\tilde{d}_H$ is as in Lemma~\ref{le:hd_ext_infty}. However, if for arbitrary $K\in\Kn$ we choose $u=\Ind_K$ and $u_j=\Ind_K+\tfrac 1j$, $j\in\N$, then $\tilde{d}(u_j,u)=+\infty$ for every $j\in\N$ while $u_j\eto u$ as $j\to+\infty$. In particular, convergence with respect to this metric is not equivalent to epi-convergence. Thus, for our purposes we need to find an extension of the Hausdorff distance that is not translation invariant anymore.

\medskip

In the following, define $\hat{d}_H:(\Kn\cup\emptyset) \times (\Kn \cup \emptyset)\to [0,\infty)$ as
$$
\hat{d}_H(K,L)=\begin{cases}
d_H(K,L)\quad & K\neq \emptyset \neq L\\
0\quad & K = L = \emptyset\\
\max\{1,d_H(K,\{0\})\}\quad& K\neq \emptyset, L= \emptyset\\
\max\{1,d_H(L,\{0\})\}\quad &K=\emptyset, L\neq \emptyset
\end{cases}
$$
for $K,L\in\Kn\cup \emptyset$.

\begin{lemma}
The functional $\hat{d}_H$ is an extension of the Hausdorff metric.
\end{lemma}
\begin{proof}
It is easy to see that $\hat{d}_H(K,L)\geq 0$ and $\hat{d}_H(K,L)=\hat{d}_H(L,K)$ for every $K,L\in\Kn\cup \emptyset$ and furthermore $\hat{d}_H(K,L)=0$ if and only if $K=L$. Moreover, $\hat{d}_H(K,L)=d_H(K,L)$ for every $K,L\in\Kn$. It remains to show the triangle inequality
\begin{equation}
\label{eq:hd_triangle_ineq}
\hat{d}_H(K,L) \leq \hat{d}_H(K,M)+\hat{d}_H(M,L).
\end{equation}
In case $K,L,M\in\Kn$ this trivially follows from the corresponding property of the Hausdorff metric. Furthermore, the cases $K=L=M=\emptyset$ as well as $K=L=\emptyset, M\neq \emptyset$ and $M=K=\emptyset, L\neq \emptyset$ are trivial. We will distinguish between the remaining cases.
\begin{itemize}
	\item $M=\emptyset, K\neq\emptyset\neq L$: In this case \eqref{eq:hd_triangle_ineq} follows from the triangle inequality for the Hausdorff metric since 
	\begin{align*}
	\hat{d}_H(K,L) &= d_H(K,L)\\
	&\leq d_H(K,\{0\}) + d_H(\{0\},L)\\
	&\leq \max\{1,d_H(K,\{0\})\} + \max\{1,d_H(L,\{0\})\}\\
	&= \hat{d}_H(K,\emptyset) + \hat{d}_H(\emptyset,L).
	\end{align*}
	\item $K=\emptyset, M\neq\emptyset\neq L$: In this case we need to show
	$$\max\{1,d_H(L,\{0\})\}\leq \max\{1,d_H(M,\{0\})\} + d_H(M,L).$$
	If $d_H(L,\{0\})\leq 1$, then the inequality above is trivial. In the remaining case $d_H(L,\{0\})> 1$, the inequality follows again from the triangle inequality for the Hausdorff metric since
	\begin{align*}
	\max\{1,d_H(L,\{0\})\} &=d_H(L,\{0\})\\
	&\leq d_H(M,\{0\}) + d_H(M,L)\\
	&\leq \max\{1,d_H(M,\{0\})\} + d_H(M,L).
	\end{align*}
\end{itemize}
Thus, we have proved \eqref{eq:hd_triangle_ineq} and $\hat{d}_H$ is an extension of the Hausdorff metric.
\end{proof}

\begin{remark}
\label{re:hd_conv_lvl_sets}
It is easy to see from the definition of $\hat{d}_H$ that for a sequence $K_j\in\Kn \cup \emptyset$ one has $\hat{d}_H(K_j,\emptyset)\to 0$ if and only if there exists $j_0\in\N$ such that $K_j=\emptyset$ for every $j\geq j_0$. Therefore, we can reformulate Lemma~\ref{le:hd_conv_lvl_sets} as follows: If $u_k, u\in\CVc$ are such that $u_k\eto u$, then $\hat{d}_H(\{u_k\leq t\},\{u\leq t\})\to 0$ as $k\to\infty$ for every $t\neq \min_{x\in\Rn} u(x)$.
\end{remark}

We will now introduce a new functional on $\CVc$. Let $\zeta:\R\to[0,\infty)$ be a continuous, decreasing function that has finite moment of order $0$, i.e., $\int_0^\infty \zeta(t)\d t <+\infty$. We define $\delta_{\zeta}^H:\CVc\times \CVc \to [0,\infty)$ as
$$\delta_\zeta^H(u,v)=\int_0^{+\infty} \hat{d}_H(\{\zeta\circ u \geq s\},\{\zeta \circ v \geq s\}) \d s$$
for $u,v\in\CVc$. In Lemma~\ref{le:delta_zeta_h_is_a_metric} we will show that this defines a metric under additional assumptions on the function $\zeta$. However, unlike the metric $\delta_{\zeta,p}$ from Section~\ref{se:sym_diff_metric_cvx_fcts}, this functional is not translation invariant, i.e.,
$$\delta_\zeta^H(u,v)\neq \delta_\zeta^H(u(\cdot-x_0),v(\cdot-x_0))$$
in general, for $x_0\in\Rn\backslash\{0\}$.

\begin{example}
Let $u=\Ind_K+s$ and $v=\Ind_L+t$ with $K,L\in\Kn$ and $s,t\in\R$, $s\leq t$. It is easy to see that
\begin{align*}
\delta_\zeta^H(u,v)&=(\zeta(s)-\zeta(t))\,\hat{d}_H(K,\emptyset)+\zeta(t)\,\hat{d}_H(K,L)\\
&=(\zeta(s)-\zeta(t))\,\max\{1,d_H(K,\{0\})\}+\zeta(t)\,d_H(K,L).
\end{align*}
In particular, if $\zeta(s)\neq\zeta(t)$, then this expression is not invariant under joint translations of $K$ and $L$.
\end{example}

We will need the following result. Recall that $h(K,z)=\max_{x\in K} \langle z,x \rangle$ denotes the support function of $K\in\Kn$ at $z\in\Sph$. Furthermore, we set $h(\emptyset,\cdot)\equiv 0$ throughout the following.

\begin{lemma}[\!\!\cite{colesanti_ludwig_mussnig_mink}, Lemma 7.1]
\label{le:support_function_intable}
If $\zeta:\R\to[0,\infty)$ is a continuous, decreasing function such that $\int_0^{\infty} \zeta(t)\d t <+\infty$, then
$$\left|\int_0^{+\infty} h(\{\zeta \circ u \geq s\},z) \d s \right| <+\infty$$
for every $u\in\CVc$ and $z\in\mathbb{S}^{n-1}$.
\end{lemma}

\begin{lemma}
\label{le:epi_conv_implies_lvl_hd_conv}
Let $\zeta:\R\to[0,\infty)$ be  a continuous, decreasing function such that $\int_0^\infty \zeta(t)\d t <+\infty$. If $u_k,u\in\CVc$ are such that $u_k\eto u$, then $\delta_\zeta^H(u_k,u)\to 0$ as $k\to \infty$.
\end{lemma}
\begin{proof}
By Lemma~\ref{le:uniform_cone} there exist $a,b\in\R$ with $a>0$ such that $u_k(x) \geq a|x|+b$ for every $k\in\N$ and $x\in\Rn$. Define $v(x)=a|x|+b-\tfrac 1a$. The set $\{v\leq t\}$ is a ball with radius greater or equal to 1 for every $t\geq b$ that furthermore contains all sets $\{u_k\leq t\}$ and $\{u\leq t\}$. Hence,
$$
0\leq \hat{d}_H(\{\zeta\circ u_k \geq s\}, \{\zeta \circ u \geq s \}) \leq \operatorname{diam}(\{\zeta \circ v \geq s\})=2\,h(\{\zeta\circ v \geq s \},z)
$$
for every $s\geq 0$ and arbitrary $z\in\Sph$. Note that the last expression is integrable over $s\geq 0$ by Lemma~\ref{le:support_function_intable}. Furthermore, by Lemma~\ref{le:hd_conv_lvl_sets} and Remark~\ref{re:hd_conv_lvl_sets},
$$\hat{d}_H(\{\zeta\circ u_k \geq s\}, \{\zeta \circ u \geq s\}) \to 0$$
for a.e. $s>0$ as $k\to\infty$. Hence, we may apply the dominated convergence theorem to obtain
$$\lim\nolimits_{k\to+\infty} \delta_\zeta^H(u_k,u)=\lim\nolimits_{k\to+\infty} \int_0^{+\infty} \hat{d}_H(\{\zeta\circ u_k \geq s\}, \{\zeta \circ u \geq s\}) \d s = 0,$$
which completes the proof.
\end{proof}

It is easy to see that we need further restrictions on $\zeta$ in order for $\delta_\zeta^H$ to become a metric. In the following, let
$$\MomZ=\left\{\zeta:\R\to(0,\infty)\,:\, \zeta \text{ is continuous, strictly decreasing and } \int_0^\infty \zeta(t) \d t < +\infty\right\}.$$

\begin{lemma}
\label{le:delta_zeta_h_is_a_metric}
For $\zeta\in \MomZ$ the functional $\delta_\zeta^H$ defines a metric on $\CVc$.
\end{lemma}
\begin{proof}
By the properties of $\zeta$ it follows from Lemma~\ref{le:epi_conv_implies_lvl_hd_conv} and its proof that $\delta_\zeta^H(u,v)<+\infty$ for every $u,v\in\CVc$. Furthermore, it is easy to see that $\delta_{\zeta}^H(u,v)=\delta_{\zeta}^H(v,u)\geq 0$ for all $u,v\in\CVcn$ and $\delta_{\zeta}(u,u)=0$. Moreover, the triangle inequality easily follows from the corresponding property of $\hat{d}_H$. Next, let $u,v\in\CVc$ be such that $\delta_{\zeta}^H(u,v)=0$. Since $\zeta$ is strictly decreasing and continuous, the maps
$$s \mapsto \{\zeta\circ u \geq s\} \quad \text{and} \quad s \mapsto \{\zeta\circ v \geq s\}$$
are continuous on $(0,\infty)$ for $s\neq \max_{x\in\Rn} (\zeta\circ u)(x)$ and $s \neq \max_{x\in\Rn} (\zeta \circ v)(x)$, respectively. Thus, if there exist $s_0,\varepsilon>0$, $\max_{x\in\Rn} (\zeta\circ u)(x) \neq s_0 \neq \max_{x\in\Rn} (\zeta\circ v)(x)$, such that
$$\hat{d}_H(\{\zeta\circ u \geq s_0\}, \{\zeta\circ v\geq s_0\})=\varepsilon>0$$
then also
$$\hat{d}_H(\{\zeta\circ u \geq s\}, \{\zeta\circ v\geq s\})>\frac{\varepsilon}{2}$$
for every $s\in(s_1,s_0]$ with some $s_1<s_0$, which implies
$$
\delta_\zeta^H(u,v)\geq \int_{s_1}^{s_0} \hat{d}_H(\{\zeta\circ u \leq s\},\{\zeta\circ v \leq s\}) \d s> \frac{\varepsilon(s_0-s_1)}{2}.
$$
Therefore, $\{\zeta \circ u \geq s\}=\{\zeta \circ v\geq s\}$ for a.e. $s>0$. Since $\zeta$ is strictly decreasing and thus invertible and since the functions are lower semicontinuous, this implies that $\{u\leq t\}=\{v\leq t\}$ for every $t\in\R$ and therefore $u\equiv v$.
\end{proof}

We will need the following result in order to show that convergence with respect to $\delta_\zeta^H$ implies epi-convergence.

\begin{lemma}
\label{le:delta_zeta_h_min_conv}
Let $\zeta\in \MomZ$. If $u_k,u\in\CVc$ are such that $\delta_\zeta^H(u_k,u)\to 0$, then\linebreak$\min_{x\in\Rn} u_k(x) \to \min_{x\in\Rn} u(x)$ as $k\to\infty$.
\end{lemma}
\begin{proof}
Let $m:=\min_{x\in\Rn} u(x)$ and $m_k:=\min_{x\in\Rn} u_k(x)$, $k\in\N$, and
assume that the statement is not true. Thus, there exist $c>0$ and a subsequence $u_{k_j}$ such that $|m_{k_j}-m|> c$ for every $j\in\N$. By possibly restricting to another subsequence, we will furthermore assume that
$$m+c< m_{k_j}$$
for every $j\in\N$ and remark that the case $m_{k_j}+c< m$ can be treated similarly. We now have $\{u_{k_j}\leq t\}=\emptyset$, $\{u\leq t\}\neq \emptyset$ and therefore
$$\hat{d}_H(\{u_{k_j}\leq t\}, \{u\leq t\}) = \hat{d}_H(\emptyset, \{u\leq t\}) \geq 1$$
for every $t\in[m,m+c]$ and every $j\in\N$. Thus, by the definition of $\delta_\zeta^H$,
\begin{align*}
\delta_{\zeta}^H(u_{k_j},u) &\geq \int_{\zeta(m+c)}^{\zeta(m)} \hat{d}_H(\{\zeta\circ u_{k_j}\geq s\},\{\zeta\circ u\geq s\}) \d s\\
&=\int_{\zeta(m+c)}^{\zeta(m)} \hat{d}_H(\{u_{k_j}\leq \zeta^{-1}(s)\},\{u\leq \zeta^{-1}(s)\})\d s\\
&\geq \zeta(m)-\zeta(m+c)\\
&>0
\end{align*}
for every $j\in\N$, which is a contradiction. Hence, $\lim_{k\to\infty} \min_{x\in\Rn} u_k(x)=\min_{x\in\Rn} u(x)$.
\end{proof}

\begin{proposition}
Let $\zeta\in \MomZ$. If $u_k,u\in \CVc$ are such that $\delta_{\zeta}^H(u_k,u)\to 0$, then $u_k\eto u$ as $k\to\infty$.
\end{proposition}
\begin{proof}
By Lemma~\ref{le:hd_conv_lvl_sets} and Remark~\ref{re:hd_conv_lvl_sets} it is enough to show that $\hat{d}_H(\{u_k\leq t\},\{u\leq t\}) \to 0$ for every $t\neq \min_{x\in\Rn} u(x)$. Assume on the contrary that there exists $t_0\neq \min_{x\in\Rn} u(x)$ such that $\hat{d}_H(\{u_k\leq t_0\},\{u\leq t_0\})\not\to 0$. In this case, there exist $c>0$ and a subsequence $u_{k_j}$ such that
$$\hat{d}_H(\{u_{k_j}\leq t_0\},\{u\leq t_0\})\geq c$$
for every $j\in\N$. By Lemma~\ref{le:delta_zeta_h_min_conv} there exists $m\in\R$ such that $t_0>m >\min_{x\in\Rn} u(x)$ and furthermore $t_0>m> \min_{x\in\Rn} u_{k_j}(x)$ for every $j\geq j_0$ with some $j_0\in\N$. By possibly restricting to another subsequence we can therefore either assume that there exists $x_j \in \{u_{k_j}\leq t_0\}$ such that $d_H(x_j,\{u\leq t_0\})\geq c$ for every $j\in\N$ or there exists $x_j\in\{u\leq t_0\}$ such that $d_H(x_j,\{u_{k_j}\leq t_0\})\geq c$ for every $j\in\N$.

Assume first that there exists $x_j\in\{u_{k_j}\leq t_0\}$ with $d_H(x_j,\{u\leq t_0 \})\geq c$ for every $j\in\N$. Since $t\mapsto\{u\leq t\}$ is continuous for $t>\min_{x\in\Rn} u(x)$, there exists $t_1>t_0$ such that
$$\{u\leq t\}\subseteq \{u\leq t_0\}+\frac c2 B^n$$
for every $t\in[t_0,t_1]$. Since $x_j\in\{u_{k_j}\leq t\}$ for every $t\geq t_0$ and $j\in\N$, this implies that
$$\hat{d}_H(\{u_{k_j}\leq t\},\{u\leq t\}) \geq \frac c2$$
for every $t\in[t_0,t_1]$. Hence,
\begin{align*}
\delta_\zeta^H(u_{k_j},u) &= \int_0^{\infty} \hat{d}_H({\{\zeta\circ u_{k_j} \geq s\},\{\zeta\circ u_\geq s\}})\d s\\
&\geq \int_{\zeta(t_1)}^{\zeta(t_0)} \hat{d}_H({\{u_{k_j} \leq \zeta^{-1}(s)\},\{ u \leq \zeta^{-1}(s)\}})\d s\\
&\geq (\zeta(t_0)-\zeta(t_1))\, \frac c2\\
&>0
\end{align*}
for every $j\in\N$, which contradicts $\delta_{\zeta}^H(u_k,u)\to 0$.

In the remaining case we assume that there exists $x_j\in\{u\leq t_0\}$ such that $d_H(x_j,\{u_{k_j}\leq t_0\})\geq c$ for every $j\in\N$. Again, since $t\mapsto \{u\leq t\}$ is continuous for $t>\min_{x\in\Rn} u(x)$, there exists $t_1\in(m,t_0)$ such that
$$\emptyset \neq B(x_j,\tfrac c2)\cap \{u\leq t\}$$
 for every $t\in[t_1,t_0]$. Since $\{u_{k_j}\leq t\} \subseteq \{u_{k_j}\leq t_0\}$ for every $t\leq t_0$ this implies that
$$\hat{d}_H(\{u_{k_j}\leq t\}, \{u\leq t\})\geq \frac c2$$
for every $t\in[t_1,t_0]$. Similar to the first case, this contradicts $\delta_\zeta^H(u_k,u)\to 0$.
\end{proof}

\subsection{Another Hausdorff-Type Metric}
\label{subse:another_hd_type_metric}
We try to find a functional analog of the usual definition of the Hausdorff metric which uses addition of balls, by replacing balls with suitable convex functions and Minkowski addition by inf-convolution. For $\lambda>0$ let $n_\lambda \in \CVc$ be defined as
$$n_\lambda(x)=h\big(\tfrac 1\lambda B^n,x\big)=\frac{1}{\lambda} |x|$$
for $x\in\Rn$. Observe, that $n_\lambda \eto \Ind_{\{0\}}$ as $\lambda\to 0^+$ and $n_\lambda^*=\Ind_{\frac 1\lambda B^n}$ for every $\lambda>0$. We now define $\delta:\CVc\times\CVc\to [0,\infty)$ as
\begin{align*}
\delta(u,v)&=\inf \{\lambda >0\colon u \geq v \infconv (n_\lambda-\lambda) \text{ and } v \geq u \infconv (n_\lambda-\lambda) \}\\
&=\inf\{\lambda >0\colon u^* \leq v^* + \Ind_{\frac 1\lambda B^n} + \lambda \text{ and } v^* \leq u^* + \Ind_{\frac 1\lambda B^n} + \lambda\}\\
&=\inf\{\lambda >0 \colon \|u^*-v^*\|_{\infty, \frac 1\lambda B^n} \leq \lambda\}
\end{align*}
for $u,v\in\CVc$, where
$$\|u^*-v^*\|_{\infty, \frac 1\lambda B^n}=\sup\nolimits_{x\in \frac 1\lambda B^n} |u^*(x)-v^*(x)|.$$ 
Note, that $u^*$ and $v^*$ are again lower semicontinuous convex functions that contain the origin in the interiors of their respective domains (see Lemma \ref{le:coercive_super_coercive_duals}). Furthermore, the only possible discontinuities of $u^*$ and $v^*$ occur at the boundary of their domains, where the functions jump to $+\infty$. Hence
$$\|u^*-v^*\|_{\infty,\frac 1\lambda B^n}$$
is continuous (and decreasing) for $\lambda\in(\lambda_0,+\infty)$ for some $\lambda_0\geq 0$. Furthermore, as $\lambda\to \lambda_0^+$ this expression either goes to $+\infty$ or it converges to a finite value with a jump to $+\infty$ as soon as $\lambda<\lambda_0$ (presuming $\lambda_0\neq 0$). In particular, this means that the infimum in the definition of $\delta(u,v)$ is in fact either a minimum or $0$.

\begin{example}
Let $u=\Ind_K$ and $v=\Ind_L$ with $K,L\in\Kn$. It follows from the definition of $\delta$ that
\begin{align*}
\delta(u,v)&=\inf \{\lambda > 0\colon \|h_K-h_L\|_{\infty,\frac 1\lambda B^n}\leq \lambda \}\\
&=\inf \{\lambda >0 \colon \tfrac 1\lambda d_H(K,L)\leq \lambda\}\\
&= \sqrt{d_H(K,L)}.
\end{align*}
\end{example}

\begin{lemma}
\label{le:delta_is_a_metric}
The functional $\delta$ defines a finite metric on $\CVc$. Furthermore, $\delta$ is invariant under joint translations of epigraphs, that is,
$$\delta(u,v)=\delta(u(\cdot-x_0)+t_0,v(\cdot-x_0)+t_0)$$
for every $u,v\in\CVc$ and for every $x_0\in\Rn$ and $t_0\in\R$.
\end{lemma}
\begin{proof}
Let $u,v\in\CVc$. It is easy to see that $\delta(u,v)=\delta(v,u)\geq 0$ and $\delta(u,u)=0$. Since the origin is contained in the interiors of $\dom u^*$ and $\dom v^*$, it is furthermore clear that this expression is finite. If $\delta(u,v)=0$, then $\|u^*-v^*\|_{\infty, \frac 1\lambda B^n} \leq \lambda$ for every $\lambda >0$. Hence, $u^*= v^*$ and therefore $u=u^{**}=v^{**}=v$. Since
$$u\geq v \infconv (n_\lambda - \lambda)$$
if and only if
$$\epi u \subseteq \epi v + \epi (n_\lambda-\lambda)$$
for $\lambda>0$, $x_0\in\Rn$ and $t_0\in\R$, it easily follows that $\delta$ is invariant under joint translations of epigraphs.

It remains to prove the triangle inequality. Fix arbitrary $u,v,w\in\CVc$. In order to show
$$\delta(u,v)\leq \delta(u,w)+\delta(w,v)$$
denote $a=\delta(u,v)$, $b=\delta(u,w)$ and $c=\delta(w,v)$. Without loss of generality we may assume that $b>0$ and $c>0$, since otherwise the statement is trivial. We now have
$$
\|u^*-w^*\|_{\infty,\frac{1}{b+c} B^n} +\|w^*-v^*\|_{\infty,\frac{1}{b+c} B^n} \leq \|u^*-w^*\|_{\infty,\frac 1b B^n} + \|w^*-v^*\|_{\infty,\frac 1c B^n} \leq b+c
$$
and therefore
\begin{equation}
\label{eq:inf_lambda_triangle_ineq}
\inf \{\lambda > 0 \colon \|u^*-w^*\|_{\infty,\frac 1\lambda B^n} + \|w^*-v^*\|_{\infty,\frac 1\lambda B^n} \leq \lambda\} \leq b+c.
\end{equation}
Furthermore, for any $\lambda > 0$,
$$\|u^*-v^*\|_{\infty, \frac 1\lambda B^n} \leq \|u^*-w^*\|_{\infty, \frac 1\lambda B^n} + \|w^*-v^*\|_{\infty,\frac 1\lambda B^n}$$
and consequently
$$\{\lambda>0\colon \|u^*-v^*\|_{\infty,\frac 1\lambda B^n}\leq \lambda\} \supseteq \{\lambda>0\colon \|u^*-w^*\|_{\infty,\frac 1\lambda B^n} + \|w^*-v^*\|_{\infty,\frac 1\lambda B^n} \leq \lambda \}.$$
Together with \eqref{eq:inf_lambda_triangle_ineq} we now have
\begin{align*}
a&=\inf \{\lambda>0\,:\, \|u^*-v^*\|_{\infty,\frac 1\lambda B^n}\leq \lambda\}\\
&\leq \inf \{\lambda>0\,:\, \|u^*-w^*\|_{\infty,\frac 1\lambda B^n} + \|w^*-v^*\|_{\infty,\frac 1\lambda B^n} \leq \lambda \}\\
&\leq b+c
\end{align*}
and therefore $\delta(u,v)\leq \delta(u,w)+\delta(w,v)$.
\end{proof}

It is easy to see that $n_\lambda\eto n_1$ as $\lambda \to 1^+$. On the other hand, since for every $\lambda>1$ we have
$$\|n_\lambda^* - n_1^*\|_{\infty,\frac 1\rho B^n}=\|\Ind_{\frac 1\lambda B^n} - \Ind_{B^n}\|_{\infty,\frac 1\rho B^n}=\begin{cases}+\infty\quad &\text{if } 0<\rho < \lambda\\
0\quad &\text{if } \rho\geq \lambda \end{cases}$$
it follows that
$$\delta(n_\lambda,n_1)=\inf\{\rho >0 \colon  \|\Ind_{\frac 1\lambda B^n} - \Ind_{B^n}\|_{\infty,\frac 1\rho B^n} \leq \rho\} = \lambda >1$$
for every $\lambda>1$. In particular, this shows that on $\CVc$, epi-convergence does not necessarily imply convergence with respect to $\delta$. However, both notions of convergence coincide on the smaller space of super-coercive convex functions.

\begin{lemma}
If $u_k,u\in\CVc$ are such that $\delta(u_k,u)\to 0$, then $u_k\eto u$ as $k\to\infty$. Furthermore, if $u_k,u\in\CVs$, then $u_k\eto u$ implies $\delta(u_k,u)\to 0$ as $k\to\infty$.
\end{lemma}
\begin{proof}
If $\delta(u_k,u)\to 0$, then $u_k^*$ converges uniformly to $u^*$ on every compact subset of $\Rn$. By Lemma~\ref{le:epi_conv_equiv} this implies $u_k^*\eto u^*$ which by Theorem~\ref{thm:epi_conv_conjugate} is equivalent to $u_k \eto u$.

If $u_k, u$ are furthermore super-coercive, then their conjugates are finite functions by Lemma~\ref{le:coercive_super_coercive_duals}. In that case, $u_k^*\eto u^*$  is equivalent to uniform convergence of $u_k^*$ to $u^*$ on every compact subset of $\Rn$, which implies $\delta(u_k,u)\to 0$.
\end{proof}

\begin{remark}
If $u,v\in\CVs$, then $u^*$ and $v^*$ are finite convex functions by Lemma~\ref{le:coercive_super_coercive_duals}, which implies that $\|u^*-v^*\|_{\infty,\frac 1\lambda B^n}$ is continuous and decreasing in $\lambda$. Therefore, in this case $\delta(u,v)=\lambda$ where $\lambda$ is uniquely determined by $\|u^*-v^*\|_{\infty,\frac 1\lambda B^n} = \lambda$.
\end{remark}

\section*{Acknowledgments}
This publication is part of a project that has received funding from the European Research Council (ERC) under the European Union’s Horizon 2020 research and innovation programme (grant agreement No. 770127). Furthermore, the authors are most grateful to Shiri Artstein-Avidan for helpful discussions and remarks.

\footnotesize

\phantomsection

\bigskip\bigskip
\parindent 0pt\normalsize

\parbox[t]{8.5cm}{
School of Mathematical Sciences\\
Tel Aviv University\\
69978 Tel Aviv, Israel\\
liben@mail.tau.ac.il\\
mussnig@gmail.com}

\end{document}